\documentclass[11pt]{article}
\usepackage{amsthm,amsmath,amssymb}
\usepackage{tikz}
\usetikzlibrary{calc}

\usepackage[colorlinks=true,citecolor=black,linkcolor=black,urlcolor=blue]{hyperref}

\theoremstyle{plain}
\newtheorem{theorem}{Theorem}[section]
\newtheorem{lemma}[theorem]{Lemma}
\newtheorem{corollary}[theorem]{Corollary}
\newtheorem{proposition}[theorem]{Proposition}

\newtheorem{claim}[theorem]{Claim}

\theoremstyle{definition}

\newtheorem{conjecture}[theorem]{Conjecture}

\theoremstyle{remark}

\title{\bf The $r$-matching sequencibility of complete multi-$k$-partite $k$-graphs}

\author{Adam Mammoliti \\
\small School of Mathematics and Statistics\\[-0.8ex]
\small UNSW Sydney\\[-0.8ex]
\small NSW 2052, Australia\\
\small\tt adam.mammoliti@outlook.com.au\\
}

\date{ 
}

\begin{document}
\maketitle

\begin{abstract}
Alspach [{\sl Bull. Inst. Combin. Appl.}~{\bf 52} (2008), 7--20] defined
the maximal matching sequencibility of a graph $G$, denoted~$ms(G)$,
to be the largest integer $s$ for which
there is an ordering of the edges of $G$ such that
every $s$ consecutive edges form a matching.
In this paper,
we consider the natural analogue for hypergraphs of this and related results
and determine $ms(\lambda\mathcal{K}_{n_1,\ldots, n_k})$
where $\lambda\mathcal{K}_{n_1,\ldots, n_k}$ denotes
the multi-$k$-partite $k$-graph with edge multiplicity $\lambda$
and parts of sizes $n_1,\ldots,n_k$, respectively.
It turns out that these invariants may be given surprisingly precise and somewhat elegant descriptions, in a much more general setting.
\end{abstract}


\section{Introduction}

Alspach~\cite{MR2394738} defined the
(maximal) {\em matching sequencibility} of a graph $G$,
denoted $ms(G)$, to be the maximum integer $s$ such that
there exist an ordering of $G$'s edges so that each $s$
consecutive edges form a matching.
Alspach~\cite{MR2394738} determined the value of $ms(K_n)$,
as follows.

\begin{theorem}\label{thm:Matching sequence} 
For an integer $n \geq 3$,
\[
ms(K_n) = \left\lfloor \frac{n-1}{2} \right\rfloor \,.
\]
\end{theorem}

Katona~\cite{MR2181045} implicitly considered
the {\em cyclic matching sequencibility} $cms(G)$ of a graph $G$
which is the natural analogue of the matching sequencibility for $G$
when cyclic orderings are allowed.
Brualdi, Kiernan, Meyer and Schroeder~\cite{MR2961987}
defined this invariant explicitly
and proved the cyclic analogue of Theorem~\ref{thm:Matching sequence}, below,
thus strengthening a weaker result by Katona~\cite{MR2181045}.

\begin{theorem}[Brualdi et al.~\cite{MR2961987}]\label{thm:Cyclic matching sequence}
For an integer $n \geq 4$,
\[
cms(K_n) = \left\lfloor \frac{n-2}{2} \right\rfloor \,.
\]
\end{theorem}

Let $K_{n,m}$ be the complete bipartite graph with
parts of cardinality $n$ and $m$.
Brualdi et al.~\cite{MR2961987} also found the matching and cyclic matching
sequencibility of complete bipartite graphs, as follows.
\begin{theorem}\label{thm: bipartite}
For integers $n$ and $m$ with $2 \leq n \leq m$,
\[
ms(K_{n,m}) = cms(K_{n,m}) =
\begin{cases}
n &\textrm{ if } n < m\,; \\
n-1 &\textrm{ if } n = m \,.
\end{cases}
\]
\end{theorem}

The aim of this paper is to generalise
Theorem~\ref{thm: bipartite} considerably with respect to
a more general notion of matching sequencibility and a more general notion of graphs.
It turns out that the resulting invariants may be given surprisingly precise and somewhat elegant descriptions; see Theorem~\ref{thm: hypergraph gen r} below.
We will consider the following generalisation of matching sequencibility
given in~\cite{MR3761920}.
For a graph $G$, $ms_r(G)$ denotes the analogue
of $ms(G)$ where consecutive edges are required to form
a graph with maximal vertex degree at most $r$.
Similarly, $cms_r(G)$ is defined in analogy to $ms_r(G)$
where we allow cyclic orderings of $G$'s edges.
A hypergraph $\mathcal{H}$ is a pair $(V,E)$
where $V$ is a set
and   $E$ is a multiset of subsets of~$V$.
The complete $k$-partite $k$-graph with parts of cardinalities $n_1, \ldots , n_k$,
denoted $\mathcal{K}_{n_1, \ldots, n_k}$, is the hypergraph
whose vertex set is the union of disjoint sets $N_1, \ldots, N_k$ of cardinalities
$n_1, \ldots , n_k$, respectively, and whose edge set is
the family of every $k$-set containing exactly one member of $N_1,\ldots,N_k$, respectively.
For a hypergraph $\mathcal{H} = (V,E)$, we let $ms_r(\mathcal{H})$ and $cms_r(\mathcal{H})$
denote the natural analogues of $ms_r(G)$ and $cms_r(G)$ for hypergraphs, respectively.
Furthermore, for any positive integer $\lambda$,
let $\lambda \mathcal{H}$ be the hypergraph $\mathcal{H}' = (V,E')$
where $E'$ contains $\lambda$ distinct copies of $e$ for each $e\in E$.
For $r\geq\Delta(\mathcal{H})$, the maximal vertex degree of~$\mathcal{H}$,
these invariants trivially equal $|E(\mathcal{H})|$.
We will extend the above definitions of $ms_r(\mathcal{H})$ and $cms_r(\mathcal{H})$ for
$r < \Delta(\mathcal{H})$ to non-trivial definitions of these invariants for $r\geq 1$.
However, the details are technical and will be given later, in Subsection~\ref{subsec:Nontriv defs}.

The main result of this paper is the following theorem,
which succeeds, perhaps surprisingly, to precisely describe the values of
 $ms_r(\lambda\mathcal{K}_{n_1,\ldots, n_k})$ and
$cms_r(\lambda\mathcal{K}_{n_1,\ldots, n_k})$.

\begin{theorem} \label{thm: hypergraph gen r}
Let
$1 \leq n_1 = n_2 = \cdots = n_u < n_{u+1} \leq \cdots  \leq n_k $
and $r = r_1 \lambda \prod_{i=2}^{k} n_i +r_2$,
for non-negative integers $r_1,r_2$ with $0 \leq r_2 \leq \lambda\prod_{i=2}^{k} n_i -1 $.
Then
\[
ms_r(\lambda\mathcal{K}_{n_1,\ldots,n_k}) =
\begin{cases}
rn_1 &\text{ if } n_1^{u-1}\mid r_2
\text{ or } (\ref{eqn: cond hypergraph comp}) \text{, below, holds}\,; \\
rn_1-1 &\text{ otherwise}\,,
\end{cases}
\]
and
\[
\hspace{-97pt} cms_r(\lambda\mathcal{K}_{n_1,\ldots,n_k}) =
\begin{cases}
rn_1 &\text{ if } n_1^{u-1}\mid r_2;
\\
rn_1-1 &\text{ otherwise}\,,
\end{cases}
\]
where
\begin{equation}\label{eqn: cond hypergraph comp}
\left(\left\lfloor \frac{r_2}{n_1^{u-1}}\right\rfloor +1 \right)
\left\lfloor \frac{\lambda}{r_2}\prod_{i=2}^{k} n_i \right\rfloor
\leq  \lambda\prod_{i=u+1}^{k}n_i
\leq \left\lfloor \frac{r_2}{n_1^{u-1}} \right\rfloor
\left(\left\lfloor \frac{\lambda}{r_2}\prod_{i=2}^{k} n_i \right\rfloor +1 \right)\,.
\end{equation}
\end{theorem}
Theorem~\ref{thm: hypergraph gen r} includes Theorem~\ref{thm: bipartite} as a special case,
which is more evident from Theorem~\ref{thm: hypergraph gen r} when $r=1$, given below.
\begin{corollary}
Let $n_1 \leq n_2 \leq \cdots \leq n_k$. Then
\[
ms(\lambda\mathcal{K}_{n_1,\ldots,n_k}) = cms(\lambda\mathcal{K}_{n_1,\ldots,n_k})=
\begin{cases}
rn_1 &\text{ if } n_1 < n_2 \,; \\
rn_1-1 &\text{ otherwise}\,.
\end{cases}
\]
\end{corollary}

Section~\ref{sec:preliminaries} contains definitions and auxiliary results. The rest of the paper is mostly dedicated to proving Theorem~\ref{thm: hypergraph gen r}.
The proof of Theorem~\ref{thm: hypergraph gen r} is
divided into three technical sections and a concluding section,
namely, Sections~\ref{sec: outline and reduction of thm}-\ref{sec: proof conclusion}.
Section~\ref{sec: Con} concludes the paper with examples of interest to
the auxiliary results in Section~\ref{sec:preliminaries}
as well as a conjecture on the value of $ms(K_{s(n)})$ and $cms(K_{s(n)})$
for complete multi-partite graphs $K_{s(n)}$.

\section{Preliminary definitions and auxiliary results}
\label{sec:preliminaries}
For technical reasons we will, contrary to the introduction, define hypergraphs without the use of ``multisets" in the following manner.
A hypergraph $\mathcal{H} = (V,E)$ is a pair consisting of two sets,
the set of {\em vertices} $V$ of $\mathcal{H}$ and
the set of {\em edges} $E$ of $\mathcal{H}$,
where each edge $e\in E$ has associated to it a prescribed set of vertices.
Each such associated vertex $v\in V$ is said to be {\em incident} with $e\in E$
and this is denoted by $v\in e$.
Here, the two distinct edges $e,e'\in E$ can be incident with the same set of vertices,
in which case $e$ and $e'$ are {\em parallel}.
We can thus view the edges of a hypergraph as a family of distinctly labelled sets comprising not necessarily distinct collections of vertices.

For an integer $n$, let $[n]:= \{0,1, \ldots , n-1 \}$.
An {\em ordering} or {\em labelling}
of a hypergraph $\mathcal{H} = (V,E)$ is
a bijective function $\ell : E \rightarrow [|E|]$.
The image of $e$ under $\ell$ is called the {\em label} of $e$.
A sequence of edges $e_0, \ldots, e_{s-1}$ is {\em consecutive} in
$\ell$ if the labels of $e_0, \ldots, e_{s-1}$ are consecutive integers,
respectively.
For a sequence $S$ of edges,
define $\mathcal{H}(S)$ to be the hypergraph
whose edges are those in the sequence $S$ and whose vertices are the
vertices incident with these edges.

For an ordering $\ell$ of a hypergraph $\mathcal{H}$,
let $ms_r(\ell)$ denote the maximum integer $s$ such that,
for every sequence $S$ of $s$ consecutive edges of $\ell$,
$\Delta(\mathcal{H}(S))\leq r$.
Define the {\em $r$-matching sequencibility} of $\mathcal{H}$,
denoted by $ms_r(\mathcal{H})$,
to be the maximum value of $ms_r(\ell)$ over all orderings $\ell$ of $\mathcal{H}$.
In particular, the special case $ms_1(\mathcal{H})$,
which we denote as $ms(\mathcal{H})$,
is the same invariant as presented in the Introduction.

A sequence of edges $e_0, \ldots, e_{s-1}$ of
a hypergraph $\mathcal{H} = (V,E)$ is {\em cyclically consecutive} in $\ell$
if the labels of $e_0, \ldots, e_{s-1}$ are consecutive integers modulo $|E|$, respectively.
We define $cms_r(\ell)$ and $cms_r(\mathcal{H})$ analogously to
$ms_r(\ell)$ and $ms_r(\mathcal{H})$, respectively,
where we now consider sequences of cyclically consecutive edges.
We first consider cases when $r < \Delta(\mathcal{H})$,
as the cases when $r \geq \Delta(\mathcal{H})$ are
somewhat different and will be dealt with in Subsection~\ref{subsec:Nontriv defs}.
The following lemma was presented in~\cite{MR3761920} and we shall give a proof for completeness.

\begin{lemma}\label{lem: 1 to r}
For a hypergraph $\mathcal{H}$ with ordering $\ell$
and integers $r_1,r_2$ with $r_1 r_2 < \Delta(\mathcal{H})$,
\[
r_2 \; ms_{r_1}(\mathcal{H}) \leq ms_{r_1 r_2} (\mathcal{H}) \,
\quad
\text{ and }
\quad
 r_2 \; cms_{r_1}(\mathcal{H}) \leq  cms_{r_1 r_2} (\mathcal{H})  \,.
\]
\end{lemma}
\begin{proof}

Let $\ell$ be a labelling of $\mathcal{H}$ such that
$cms_{r_1}(\ell) = cms_{r_1}(\mathcal{H})$.
Any sequence $S$ of $r_2 \; cms_{r_1}(\ell)$
cyclically consecutive edges of $\ell$
consists of $r_2$ subsequences of $ cms_{r_1}(\ell)$
cyclically consecutive edges of $\ell$ and each subsequence
forms a hypergraph for which every vertex has degree at most~$r_1$.
Thus, every vertex has degree at most $r_1 r_2$ in $\mathcal{H}(S)$.
Hence,
\[
cms_{r_1r_2}(\mathcal{H})
\geq cms_{r_1r_2}(\ell)
\geq r_2 \; cms_{r_1}(\ell)
= r_2 \; cms_{r_1}(\mathcal{H})\,.
\]
The non-cyclic case is similar and, therefore, omitted.
\end{proof}

For edge-disjoint hypergraphs $\mathcal{H}_0 ,\ldots , \mathcal{H}_{a-1}$
on the same vertex set~$V$,
with labellings $\ell_{0} ,\ldots ,\ell_{a-1}$, respectively,
let $\ell_0 \vee\cdots\vee \ell_{a-1}$ denote
the ordering $\ell$ of $G= \bigl(V ,\bigcup_{i=0}^{a-1} E(\mathcal{H}_i)\bigr)$
defined by
$\ell(e_{i,j}) = \ell_j (e_{i,j}) + \sum_{l=0}^{j-1} |E(\mathcal{H}_l)|$
where $e_{ij} \in E(\mathcal{H}_j)$ for all $i$ and $j$.
Let $s$ be an integer and
$\mathcal{H}$ and $\mathcal{H}'$ be edge-disjoint hypergraphs on the same vertex set $V$,
each having at least $s-1$ edges.
Also, let $\mathcal{H}$ and $\mathcal{H}'$
have labellings $\ell$ and $\ell'$, respectively,
and let $\mathcal{H}_s$ be the subhypergraph of
$\bigl( V,E(\mathcal{H}) \cup E(\mathcal{H}')\bigr)$
that consists of the last $s-1$ edges of $\ell$ and the first $s-1$ edges of $\ell'$.
Then we will let $\ell \vee_s \ell'$ denote the ordering of $\mathcal{H}_s$ for which
the edges of $\mathcal{H}_s$ appear in the same order as they do in $\ell \vee \ell'$.
We now define $ms_r(\ell , \ell')$
to be the largest integer $s$ such that
$ms_r(\ell \vee_{s} \ell') \geq s$.

A {\em matching} of a hypergraph $\mathcal{H}$ is a subhypergraph $\mathcal{M}$
in which every vertex has degree~$1$.
A~{\em matching decomposition} of a hypergraph $\mathcal{H} = (V,E)$ is
a set of matchings of~$\mathcal{H}$ that partition the edge set $E$.
The following proposition, presented in~\cite{MR3761920}, 
gives a lower bound on the $r$-cyclic matching sequencibility,
given that a matching decomposition with certain properties exists.
In the proposition,
the subscripts of the orderings $\ell_i$ are taken modulo $t$:
$\ell_{i+r} = \ell_{i'}$ holds exactly when $i' \equiv i+r \pmod t$.
\begin{proposition}\label{prop: Matching decomposition both}
Let $\mathcal{H}$ be a hypergraph that decomposes into matchings
$\mathcal{M}_0 ,\ldots , \mathcal{M}_{t-1}$,
each with $n$ edges and orderings $\ell_0, \ldots, \ell_{t-1}$, respectively.
Suppose, for some $x \in [n]$ and $r < \Delta(G)$,
that $ms(\ell_i ,\ell_{i+r}) \geq n-x$ for all $i\in [t-r]$.
Then $ms_r(G) \geq rn-x$,
and if $ms(\ell_i ,\ell_{i+r}) \geq n-x$ for all $i\in [t]$,
then $cms_r(G) \geq rn-x$.
\end{proposition}
\noindent
The following definitions are used here and throughout the paper.
For a hypergraph~$\mathcal{H}$ with ordering $\ell$,
$S_{\ell}(\mathcal{H})$ denotes the sequence of edges of $\mathcal{H}$ listed in
the same order as~$\ell$,
and $\ell$ {\em corresponds} to $S_{\ell}(\mathcal{H})$;
i.e., if $e_0 ,\ldots, e_{k-1}$ is a sequence of the edges of $\mathcal{H}$,
then~$\ell$ corresponds to that sequence if $\ell(e_i) = i$ for all $i \in [k]$.
We will omit the subscript~$\ell$ if the ordering is clear.
Also, for edge disjoint graphs
$\mathcal{H}_0 , \ldots ,\mathcal{H}_{a-1}$
with labellings $\ell_0 ,\ldots, \ell_{a-1}$, respectively,
one can check that the ordering $\ell = \ell_0 \vee\cdots\vee \ell_{a-1}$
corresponds to sequence
$S_{\ell_0}(\mathcal{H}_0) \vee\cdots\vee S_{\ell_{a-1}}(\mathcal{H}_{a-1})$.
Proposition~\ref{prop: Matching decomposition both} 
was proven for graphs in~\cite{MR3761920}. 
We provide the details for hypergraphs for completeness.

\begin{proof}[Proof of Proposition~\ref{prop: Matching decomposition both}]

We consider only the cyclic case, as the non-cyclic case is similar.
Let~$\ell$ be the ordering corresponding to
$S_{\ell_0}(\mathcal{M}_{0}) \vee\cdots\vee S_{\ell_{t-1}}(\mathcal{M}_{t-1})$.
Consider a sequence $S$ of $rn-x$ consecutive edges of $\ell$.
The sequence $S$ is of the form
\[
  \underbrace{e_1 ,\ldots, e_{j}}_\text{edges in $\mathcal{M}_{i}$},
  S_{\ell_{i+1}}(\mathcal{M}_{i+1}) \vee\cdots\vee S_{\ell_{i+r-a}}(\mathcal{M}_{i+r-a}),
  \underbrace{e_{j+1} ,\ldots, e_{an-x}}_\text{edges in $\mathcal{M}_{i+r+1-a}$} \,,
\]
for some $i \in [t]$, $j \in [n+1]$, and $a$.
If $S$ contains edges from only one matching $\mathcal{M}_l$,
then $S$ is a subsequence of $S_{\ell}(\mathcal{M}_l)$ and $r=1$.
Then we are done, as $\mathcal{H}(S)$ is clearly a matching.
Hence, without loss of generality,
we can assume that $S$ contains edges from each of
$\mathcal{M}_i$ and $\mathcal{M}_{i+r+1-a}$.
Let $S'$ be the sequence of the edges of $S$
which are in either $\mathcal{M}_i$ or $\mathcal{M}_{i+r+1-a}$, in order with respect to~$S$.
There are $0 < an - x \leq 2n$ edges in $S'$.
Therefore, $a = 1$ or $a=2$.

If $a=2$, then $S$ is a subsequence of
$S(\mathcal{M}_{i}) \vee \cdots \vee S(\mathcal{M}_{i+r-1})$,
and, hence, $\Delta(\mathcal{H}(S)) \leq r$.
If $a=1$, then the first $j$ edges and last $n-j-x$ edges of $S$ and thus $S'$
form the sequence of the last $j$ edges of $\ell_{i}$ and the first $n-j-x$ edges of
$\ell_{i+u+1}$, respectively.
Therefore,
the $j + n-j - x = n - x$
edges of $S'$ are consecutive in
$\ell_{i} \vee_{n - x} \ell_{i+r}$.
By assumption, $ms(\ell_i , \ell_{i+r}) \geq n_1-x$, so $\mathcal{H}(S')$ must be a matching.
The edges of $S$ not in $S'$ are from the $r-1$ matchings
$\mathcal{M}_{i+1} , \ldots ,\mathcal{M}_{i+r-1}$.
Thus, $\Delta(\mathcal{H}(S)) \leq r$.
\end{proof}

An ordering of a set $A$ is a bijective function $\sigma : A \rightarrow [|A|]$.
Many of the matching decompositions that we will use henceforth
have a natural indexing which is not directly compatible with
Proposition~\ref{prop: Matching decomposition both}.
In such cases we will find it useful to be able to find
an ordering of the set of indices, with particular properties.
To do this, we will make use of the following lemma, first given in~\cite{MR3761920}.
\begin{lemma}\label{lem: General cyclic ordering}
Let $s<t$ be integers and set $d := \gcd(s,t)$.
Define $a_{i,j} := \left( j\!\!\!\mod{\frac{t}{d}} \right)
                 +\left(i\frac{t}{d}\!\!\mod{t}\right)$
for all integers $i$ and $j$.
Then some ordering $\sigma$ of $[t]$
satisfies
$\sigma(a_{i,j+1}) = (\sigma(a_{i,j})+s) \!\!\mod{t}$
for all $i \in [d]$ and $j \in \bigl[\frac{t}{d}\bigr]$.
\end{lemma}
\begin{proof}
We check that the function
$\sigma \,:\, [t] \rightarrow [t]$ defined by
$\sigma(a_{i,j}) = (i + js) \,\,\,\text{modulo}\,\,{t}$
for $i \in [d]$ and $j \in \bigl[\frac{t}{d}\bigr]$ will suffice.
Suppose that $i+ js \equiv i'  + j's \pmod t$
for some $i,i' \in [d]$ and $j,j' \in\bigl[\frac{t}{d}\bigr]$.
Then
$i - i' \equiv (j'-j)s \pmod t$.
As $d$ divides $s$ and $t$, any multiple of $s$ modulo $t$ is also a multiple of $d$.
Thus,  $i-i'$ is a multiple of $d$, while $0 \leq |i-i'| \leq d-1$. This
is only possible if $i=i'$ and so $(j - j')s \equiv 0 \pmod t$.
As
$0 \leq |j' -j| \leq \frac{t}{d}-1$
and $\textrm{lcm}(s,t)= \frac{st}{d}$,
we must also have that $j=j'$.
Thus, $\sigma$ is injective and so bijective; $\sigma$ is thus an ordering of $[t]$.
For any
$i \in [d]$ and $j \in \bigl[\frac{t}{d}\bigr]$,
\[
    \sigma(a_{i,j+1})
  = \left(i+(j+1)s\right) \,\,\,\text{modulo}\,\,{t}
  = \left(\sigma(a_{i,j}) +s\right) \,\,\,\text{modulo}\,\,{t}\,.
\]
Hence, $\sigma$ has the required properties.
\end{proof}
\noindent
The function $\sigma$ in the lemma also satisfies an analogous non-cyclic
property, as follows.
\begin{corollary}\label{cor: General non cyclic ordering}
Let $s<t$ be integers.
Then there exists an ordering $\tau$ of $[t]$
with the property that, if $\tau(a) \leq t-s-1$,
then $\tau(a+1) = \tau(a)+s$.
\end{corollary}
\noindent
We use Lemma~\ref{lem: General cyclic ordering} to give an analogous
version of Proposition~\ref{prop: Matching decomposition both} for the cyclic case.
\begin{proposition}\label{prop: Matching decomposition}
Let $\mathcal{H}$ be a hypergraph that decomposes into matchings $\mathcal{M}_{i,j}$,
each with $n$ edges and orderings $\ell_{i,j}$ for $i \in [d]$ and $j \in [c]$, respectively.
Suppose, for some $x \in [n]$ and $r < \Delta(\mathcal{H})$,
that $\gcd(dc,r) = d$ and $ms(\ell_{i,j} ,\ell_{i,j+1}) \geq n-x$
for all $i \in [d]$ and $j \in [c]$.
Then $cms_r(\mathcal{H}) \geq rn-x$.
\end{proposition}
\begin{proof}
Let $a_{i,j}$ and $\sigma$ be as defined in Lemma~\ref{lem: General cyclic ordering}
for $s = r$ and $t = cd$.
Set $\mathcal{M}_{\sigma(a_{i,j})} := \mathcal{M}_{i,j}$
and $\ell_{\sigma(a_{i,j})} := \ell_{i,j}$
for all $i \in [d]$ and $j \in [c]$.
For $l \in [t]$, let $l = \sigma(a_{i,j})$.
By Lemma~\ref{lem: General cyclic ordering},
$\sigma(a_{i,j+1})\equiv \sigma(a_{i,j})+r \equiv l+r \pmod{t}$.
Hence, $ms(\ell_{l},\ell_{l+r}) = ms(\ell_{i,j},\ell_{i,j+1}) \geq n-x$.
Therefore, the conditions of Proposition~\ref{prop: Matching decomposition both} are
satisfied and the result follows.
\end{proof}

\noindent One could also use
Corollary~\ref{cor: General non cyclic ordering} to create an analogous version
of Proposition~\ref{prop: Matching decomposition both}
for the non-cyclic case, but we will not require this.

\subsection{Non-trivial definitions of \texorpdfstring{$ms_r(\mathcal{H})$ and $cms_r(\mathcal{H})$ for all $r\geq 1$}{}}\label{subsec:Nontriv defs}

If $\mathcal{H}$ is a hypergraph with maximum degree $\Delta(\mathcal{H})$ and
$r \geq \Delta(\mathcal{H})$, then one might say that, trivially,
$ms_r(\mathcal{H}) = |E(\mathcal{H})|$,
as clearly any sequence of edges containing all the edges of~$\mathcal{H}$
form~$\mathcal{H}$, which has no vertex of degree greater than~$r$.
Somewhat implicitly, the definition of cyclic $r$-matching sequencibility
allows $r \geq \Delta(\mathcal{H})$,
and $cms_r(\mathcal{H})$ is non-trivial in general.
However, when $r < \Delta(\mathcal{H})$, $ms_r(\mathcal{H})$ and $cms_r(\mathcal{H})$ have
the intuitive relationship $cms_r(\mathcal{H}) \leq ms_r(\mathcal{H})$ for any $\mathcal{H}$.
Thus, to preserve that relationship for all $r$ and make the determination
of $ms_r(\mathcal{H})$ for hypergraphs with $r \geq \Delta(\mathcal{H})$ of interest,
we will give a definition of $ms_r(\mathcal{H})$ which is non-trivial in general, for all $r\geq 1$.

Let $\mathcal{H} =(V,E)$ be a hypergraph with an ordering $\ell$
and, to use the notation of Bondy and Murty~\cite{MR0411988},
let $\varepsilon := |E|$.
First, recall the notion of cyclically consecutive edges.
A sequence $S= e_0, \ldots , e_{s-1}$ of edges in $E$ is {cyclically consecutive in $\ell$
if the labels of $e_0 , \ldots, e_{s-1}$ are cyclically consecutive integers
modulo $\varepsilon$, respectively.
In particular, a sequence of $s > \varepsilon$ edges can be cyclically consecutive,
where $e_{i}$ and $e_{i+\varepsilon}$ must be the same edge,
for all $i \in [s-\varepsilon]$.
We define $\mathcal{H}(S)$ to be the hypergraph with (distinctly labelled) edges $e_0 , \ldots, e_{s-1}$

We now define  $ms_r(\mathcal{H})$ for all $r\geq 1$.
For an integer $s$,
let $a$ be the integer such that $a \varepsilon \leq s < (a+1)\varepsilon$.
A sequence $e_0 \ldots , e_{s-1}$ of edges of $\mathcal{H}$ is {\em consecutive} in $\ell$
if $\ell(e_0)  \leq  (a+1)\varepsilon-s$ and
the labels of $e_0 \ldots , e_{s-1}$ are cyclically consecutive integers modulo~$\varepsilon$,
respectively.
The definition of consecutive edges, given earlier in the section,
is recovered by setting $a=0$.
Define $ms_r(\ell)$ to be the largest value $s$ such that,
for every sequence $S$ of $s$ consecutive edges in $\ell$,
$\Delta(\mathcal{H}(S))\leq r$.
Define $ms_r(\mathcal{H})$ to be the largest value of $ms_r(\ell)$
over all orderings $\ell$ of~$\mathcal{H}$.
As the edges in a sequence $S = e_0, \ldots , e_{s-1}$ of consecutive edges of $\ell$ are also
cyclically consecutive under the restriction $\ell(e_0)  \leq  (a+1)\varepsilon-s$,
it follows that $cms_r(\ell) \leq ms_r(\ell)$ and,
thus, $cms_r(\mathcal{H}) \leq ms_r(\mathcal{H})$
for all positive integers $r$ and hypergraphs $\mathcal{H}$.

We now demonstrate
that $ms_r(\mathcal{H})$, as defined above, is non-trivial in general.
For a hypergraph $\mathcal{H} =(V,E)$ and positive integer $\lambda$,
let $\lambda \mathcal{H}$
be the hypergraph $\mathcal{H}' = (V,E')$
where $E'$ is formed from $E$ by replacing each $e\in E$ with
$\lambda$ distinct edges parallel to $e$.
For an ordering $\ell$ of $\mathcal{H}$ and integer $a$,
let $a\ell := \ell \vee \cdots \vee \ell$, where $\ell$ occurs $a$ times.
That is, $a\ell$ corresponds to the sequence
$e_0, \ldots, e_{a\varepsilon - 1}$ of edges of $\mathcal{H}$ such that
$S_{\ell}(\mathcal{H}) = e_0, \ldots, e_{\varepsilon -1}$, and
$e_i$ and $e_{i+\varepsilon}$ are
the same edge for all $i \in [(a-1)\varepsilon]$.
In particular, for an integer $s$ such that $a \varepsilon \leq s < (a+1)\varepsilon$,
the set of all sequences $S$  of $s$ consecutive edges of $\ell$ is
the set of all sequences $S'$ of $s$ consecutive edges of $(a+1)\ell$.
Also, the hypergraph formed by the sequence corresponding to $b\ell$ is $b\mathcal{H}$
for all positive integers $b$.
So, for any $r$,
if $a$ is the integer such that
$a \Delta(\mathcal{H}) \leq r < (a+1) \Delta (\mathcal{H})$,
then $ms_r(\mathcal{H}) = s$ for some $s$ such that
$a \varepsilon \leq s < (a+1)\varepsilon$ and,
in general, the value $s$ is non-trivial for any $r \geq 1$ and hypergraph $\mathcal{H}$.

The two following lemmas will each be used in several parts of the proof of
Theorem~\ref{thm: hypergraph gen r}.

\begin{lemma}\label{lem: seqity gen r}
Let $\mathcal{H}$ be a hypergraph with $\varepsilon$ edges and maximum degree $\Delta$,
and $r=a\Delta +b$ for non-negative integers $a$ and $b$ with $b \in [\Delta]$.
Then
\[
  a\varepsilon+ms_b(\mathcal{H})\leq ms_r(\mathcal{H})
  \quad \text{ and } \quad
  a\varepsilon+cms_b(\mathcal{H}) \leq cms_r(\mathcal{H})\,.
\]
\end{lemma}
\begin{proof}
Let $s = a \varepsilon + ms_b(\mathcal{H})$ and $\ell$ be
an ordering of $\mathcal{H}$ satisfying $ms_b(\ell) = ms_b(\mathcal{H})$.
Consider a sequence  $S = e_0, \ldots, e_{s-1} $ of $s$ consecutive edges of $\ell$.
As $e_i = e_{i+\varepsilon}$ for all $i \in [s-\varepsilon]$,
$a+1$ copies of the edge $e_j$ occur in the sequence $S$
if $j \in [s-a\varepsilon]$,
and $a$ copies of the edge $e_j$ occur if $s-a\varepsilon \leq j \leq \varepsilon-1$.
In particular, $\mathcal{H}(S)$  is the hypergraph obtained by
adding to $a\mathcal{H}$ an edge parallel to $e$ for each edge $e$ in
the sequence $S' := e_0, \ldots, e_{s-a\varepsilon-1}$.
The sequence $S'$ is consecutive in $\ell$,
as $\ell(e_{0}) \leq (a+1)\varepsilon -s$.
Since $ms_{b}(\ell) =s-a \varepsilon$, $\Delta(\mathcal{H}(S')) \leq b$.
Thus, the degree of a vertex $v$ in $\mathcal{H}(S)$
is at most $a \deg_{\mathcal{H}}(v) + b \leq a\Delta +b = r$.
Hence, $ms_r(\mathcal{H}) \geq ms_r(\ell) \geq s = a \varepsilon + ms_b(\mathcal{H})$.
The cyclic case is similar.
\end{proof}

\begin{lemma}\label{lem: multi from non-multi}
For a hypergraph $\mathcal{H}$ and $\lambda \geq 1$,
$cms_r(\lambda \mathcal{H}) \geq cms_r(\mathcal{H})$.
\end{lemma}
\begin{proof}
Let $\ell$ be an ordering of $\mathcal{H}$ satisfying $cms_r(\ell) = cms_r(\mathcal{H})$.
For an edge $e \in E(\mathcal{H})$, let $e_0' , \ldots, e'_{\lambda-1}$ be the corresponding
edges parallel to $e$ in $E(\lambda\mathcal{H})$.
By identifying each of $e_0' , \ldots, e'_{\lambda-1}$ with
a unique copy of $e$ in the sequence $S_{\lambda \ell}(\mathcal{H})$,
we can define $\ell'= \lambda \ell$ to be an ordering of $\lambda \mathcal{H}$.
For any sequence $S$ of $s$ cyclically consecutive edges of $\ell$
and the corresponding sequence $S'$ of $s$ cyclically consecutive edges of $\ell'$,
clearly $\mathcal{H}(S) =\mathcal{H}(S')$.
Therefore, $cms_r(\ell') = cms_r( \ell)$ and, thus,
$cms_r(\lambda \mathcal{H}) \geq cms_r(\mathcal{H})$.
\end{proof}
\noindent
An analogous result to Lemma~\ref{lem: multi from non-multi}
in the non-cyclic case does not hold; see Section~\ref{sec: Con}.

\section{Proof of Theorem~\ref{thm: hypergraph gen r}: Part I}
\label{sec: outline and reduction of thm}
Theorem~\ref{thm: hypergraph gen r} will be proved by a set of lemmas
that fall into three separate categories,
each to be addressed in this and the next two sections.
The first two of these lemmas are given in the present section.

We start by introducing
the following notation, which will be
used in the remainder of the paper.
Let $\lambda \geq 1$, $1 \leq n_1 \leq \cdots \leq n_k$ and $u$ be the largest integer such that $n_1 = n_u$.
Let $N = \prod_{i=2}^{k}n_i$, $N' = \prod_{i=u+1}^{k}n_i$,
$r = r_1 \lambda N +r_2$ and $\lambda N = ar_2+b$
for integers $a,b,r_1$ and $r_2$ such that $r_2 \in [\lambda N]$ and
$b \in [r_2]$.

Recall from the Introduction that the complete $k$-partite $k$-hypergraph,
denoted by $\mathcal{K}_{n_1,\ldots , n_k}$,
is the hypergraph whose vertex set $V$ is the union of disjoint
sets $N_1, \ldots, N_{k}$ of sizes $n_1, \ldots, n_k$,
respectively, and whose edge set $E$ is the family of all $k$-edges
that have exactly one endpoint in $N_i$ for all $i$.
We note that the inequality 
$ms_r(\lambda\mathcal{K}_{n_1,\ldots,n_k}) \leq rn_1$
is trivial for all $r$ as every edge
incident with one of the $n_1$ vertices of $N_1$ and, therefore,
a sequence of at most $rn_1$ edges of
$\lambda \mathcal{K}_{n_1,\ldots,n_k}$ can form a hypergraph
with maximum degree at most $r$.
Thus, the inequalities
$cms_r(\lambda \mathcal{K}_{n_1,\ldots,n_k})
\leq ms_r(\lambda \mathcal{K}_{n_1,\ldots,n_k}) \leq rn_1$ will always hold.

The following claim is an immediate necessary condition
for an ordering $\ell$ of $\lambda \mathcal{K}_{n_1,\ldots,n_k}$
to satisfy $ms_r(\ell) = r n_1 $ or $cms_r(\ell) = r n_1$.
\begin{claim}\label{claim: hypergraph contr}
Let
$\ell$ be an ordering of $\lambda \mathcal{K}_{n_1,\ldots,n_k}$.
If $ms_r(\ell) = r n_1$,
then the edges $\ell^{-1}(j)$ and $\ell^{-1}(r_2 n_1 +j)$
are incident with the same vertex in $N_i$
for all $i=1,\ldots, u$ and $j \in [\lambda Nn_1 -r_2 n_1]$.
If $cms_r(\ell) = r n_1$,
then the edges $\ell^{-1}(j)$ and 
$\ell^{-1}\left((r_2 n_1 +j)\!\!\mod{\lambda N n_1}\right)$
are incident with the same vertex in $N_i$ for all $i=1,\ldots, u$
and $j \in [\lambda N]$.
\end{claim}

\begin{proof}
We only prove the non-cyclic case as the cyclic case is similar.
Let $\ell$ be an ordering of $\lambda\mathcal{K}_{n_1,\ldots,n_k}$
such that $ms_r(\ell) = r n_1$,
and let $\varepsilon := |E(\lambda \mathcal{K}_{n_1,\ldots,n_k})| = \lambda N n_1$.

Consider a sequence $S = e_0 ,\ldots, e_{rn_1} $ of consecutive edges of $\ell$,
where, by definition,
$j := \ell(e_0)\in [(r_1+1)\varepsilon -rn_1] = [\varepsilon -r_2n_1]$.
The sequence $S' = e_{1}, \ldots, e_{rn_1-1}$ consists of
$rn_1-1$ consecutive edges of $\ell$
and so $(\mathcal{H}(S')) \leq r$.

As every edge in $E(\lambda \mathcal{K}_{n_1,\ldots,n_k})$ is incident with
a vertex in each of $N_1, \ldots, N_u$ and $|N_i| = n_1$ for $i \leq u$,
every vertex in each of $N_1, \ldots, N_u$ must have degree
exactly $r$ in $\mathcal{H}(S')$,
except for some $v_1 \in N_1 ,\ldots , v_u \in N_u$ which each have degree $r-1$.
Thus, in order for the hypergraphs formed by the sequences
$S_0 = e_0 ,\ldots, e_{rn_1-1}$ and $S_1 = e_1 ,\ldots, e_{rn_1}$
to each have maximum degree at most $r$,
the edges $e_0$ and $e_{rn_1}$ must be incident with
each of $v_1 \in N_1 ,\ldots , v_u \in N_u$.
As $e_i = e_{i+\varepsilon}$
for all $i \in [rn_1-\varepsilon]$,
$e_{rn_1} = e_{r'}$ for
$r' = rn_1 \!\!\mod{\varepsilon}$.
Since $ r = r_1 \lambda N + r_2$ and $ \varepsilon =  \lambda Nn_1$,
it follows that $e_0$ and $e_{r'} = e_{r_2n_1}$
are incident with $v_1, \ldots, v_u$;
i.e., the edges $\ell^{-1}(j)$ and $\ell^{-1}(r_2 n_1 +j)$
are incident with the same vertex in $N_i$
for all $ i = 1, \ldots , u$, as required.
\end{proof}

\begin{lemma}\label{lem: hypergraph necessary condition}
If $ms_r(\lambda\mathcal{K}_{n_1,\ldots , n_k}) = r n_1$, then $n_1^{u-1} \mid r_2$ or
\[
\left(\left\lfloor \frac{r_2}{n_1^{u-1}}\right\rfloor  +1\right)
\left\lfloor \frac{\lambda}{r_2}\prod_{i=2}^{k} n_i\right\rfloor
\leq \lambda\prod_{i=u+1}^{k}n_i \leq
\left\lfloor \frac{r_2}{n_1^{u-1}}\right\rfloor
\left(\left\lfloor \frac{\lambda}{r_2}\prod_{i=2}^{k} n_i\right\rfloor +1 \right)\,.
\]
\end{lemma}
\begin{proof}
Let $\ell$ be an ordering of
$\lambda\mathcal{K}_{n_1,\ldots , n_k}$ such that
$ms_r(\ell)= rn_1$.
Let \\$S_{\ell}(\lambda\mathcal{K}_{n_1,\ldots , n_k}) = e_0, \ldots, e_{\lambda N n_1-1}$
and let $S = e_0' , \ldots, e_{\lambda Nn_1 -1}'$ be the sequence of
edges from $E(\mathcal{K}_{n_1,\ldots , n_u})$ such that
if $e_i$ is incident with each of $v_1 \in N_1, \ldots, v_u \in N_u$,
then $e_i'$ is the edge in $E(\mathcal{K}_{n_1,\ldots , n_u})$
incident with each of $v_1, \ldots, v_u$.
For an edge $e \in E(\mathcal{K}_{n_1,\ldots , n_u})$,
let $d(e)$ be the number of times that $e$ appears among the first $r_2n_1$
edges of $S$.
Similarly, let $d'(e)$ number of times that $e$ appears among the first $b n_1$
edges of~$S$, where $d'(e)$ is 0 for all $e$ if $b=0$.

We count in two ways the number times that
an edge $e \in E(\mathcal{K}_{n_1,\ldots , n_u})$ appears in~$S$.
For all $j \in [\lambda N n_1 -r_2 n_1]$,
Claim~\ref{claim: hypergraph contr} implies that
the edges $e_j$ and $e_{r_2 n_1 +j}$
are incident with the same vertex in $N_i$ for $i = 1, \ldots, u$.
Therefore, $e'_j = e'_{r_2 n_1+j}$ for all $j \in [\lambda N n_1 -r_2 n_1]$.
In particular, $e'_j = e'_{a r_2 n_1+j}$ for $j \in [bn_1]$,
where $[bn_1]=[0] =\emptyset$ if $b=0$.
As $\lambda N = ar_2+b$,
the first $b n_1$ edges and the last $b n_1$ edges of $S$ (in order) are therefore the same.
Thus, the edge $e \in (\mathcal{K}_{n_1,\ldots , n_u})$
appears $a d(e) + d'(e)$ times in the sequence $S$.
On the other hand, as $\ell$ is an ordering of $\lambda \mathcal{K}_{n_1,\ldots , n_k}$,
any vertices $v_1 \in N_1, \ldots, v_u \in N_u$ are incident with exactly
$\lambda N'$ edges in the sequence $S_{\ell}(\mathcal{H})$.
Thus, each edge $e \in E(\mathcal{K}_{n_1,\ldots , n_u})$
appears $\lambda N'$ times in~$S$.
Hence,
\begin{equation}\label{eqn:condition on de}
\lambda N' = a d(e) + d'(e)
\end{equation}
for all $e \in E(\mathcal{K}_{n_1,\ldots , n_u})$.

We now establish the upper inequality of the lemma.
As the first $b n_1$ edges of $S$ are contained in the first $r_2 n_1$ edges of $S$,
clearly $d'(e) \leq d(e)$ for all $e$.
So, by (\ref{eqn:condition on de}),
we have that $(a+1)d(e) \geq \lambda N'$ for all $e \in E(\mathcal{K}_{n_1,\ldots , n_u})$.
In particular,
$(a+1)d_{\min} \geq \lambda N'$,
where $d_{\min}$ is the minimum of $d(e)$
over all edges $e\in E(\mathcal{K}_{n_1,\ldots , n_u})$.
Clearly,
\begin{equation}\label{eqn: rn sum}
\sum_{e \in E(\mathcal{K}_{n_1,\ldots , n_u})}d(e) = r_2 n_1\,,
\end{equation}
and so, by the Pigeonhole Principle,
$d_{\min}
\leq \bigl\lfloor \frac{r_2 n_1}{n_1^{u}}\bigr\rfloor
 =   \bigl\lfloor \frac{r_2 }{n_1^{u-1}}\bigr\rfloor$.
Thus,
\[
(a+1)\left\lfloor \frac{r_2 }{n_1^{u-1}}\right\rfloor \geq (a+1)d_{\min}\geq \lambda N'\,,
\]
which is equivalent to
\[
\left(\left\lfloor \frac{\lambda }{r_2}\prod_{i=2}^{k} n_i \right\rfloor +1 \right)
\left\lfloor \frac{r_2}{n_1^{u-1}}\right\rfloor  \geq  \lambda \prod_{i=u+1}^{k}n_i \,.
\]
This establishes the upper inequality of the lemma.

We now establish the lower inequality of the lemma.
Since $d' (e) \geq 0$,
(\ref{eqn:condition on de}) implies that
$\lambda N' \geq a d(e)$ for all $e \in E(\mathcal{K}_{n_1, \ldots, n_u})$.
In particular, $\lambda N' \geq a d_{\max}$, where $d_{\max}$ is the maximal value
of $d(e)$ for edges $e \in E(\mathcal{K}_{n_1, \ldots, n_u})$.
By (\ref{eqn: rn sum}) and the Pigeonhole Principle,
$d_{\max} \geq \Bigl\lceil\frac{r_2}{n_1^{u-1}}\Bigr\rceil$.
Thus,
\[
\lambda \prod_{i=u+1}^{k}n_i =
\lambda N' \geq  a \left\lceil\frac{r_2}{n_1^{u-1}} \right\rceil  =
\left\lfloor \frac{\lambda}{r_2}\prod_{i=2}^{k} n_i
\right\rfloor \left\lceil \frac{r_2}{n_1^{u-1}}\right\rceil\,,
\]
which establishes the lower inequality of the lemma if $n_1^{u-1}\nmid r_2$.
Otherwise, $n_1^{u-1}\mid r_2$, and we are done.
\end{proof}

\begin{lemma}\label{lem: hypergraph cyclic necessary condition}
If $cms_r(\lambda \mathcal{K}_{n_1,\ldots , n_u}) = r n_1$,
then $n_1^{u-1} \mid r_2$.
\end{lemma}
\begin{proof}
Let $\ell$ be an ordering of $\lambda\mathcal{K}_{n_1,\ldots, n_k}$
such that $cms_r(\ell)= rn_1$.
Let $x$ and $y$ be integers satisfying $x r_2 = y \lambda N$.
Write $S_{y\ell}(\lambda\mathcal{K}_{n_1,\ldots , n_k}) = e_0, \ldots, e_{y\lambda N n_1-1}$
and let $S = e_0' , \ldots, e_{y \lambda Nn_1 -1}'$ be the sequence of
edges from $E(\mathcal{K}_{n_1,\ldots , n_u})$ such that,
if $e_{i}$ is incident with each of $v_1 \in V_1, \ldots, v_u \in V_u$,
then $e'_{i}$ is the edge in $E(\mathcal{K}_{n_1,\ldots , n_u})$
incident with each of $v_1, \ldots, v_u$.
For an edge $e \in E(\mathcal{K}_{n_1,\ldots, n_u})$
let $d(e)$ be the number of times that $e$ appears among the first $r_2 n_1$ edges of~$S$.

We count in two ways the number of times that
an edge $e \in E(\mathcal{K}_{n_1,\ldots , n_u})$ appears in~$S$.
For all $j$,
Claim~\ref{claim: hypergraph contr} implies that
edges $e_j$ and $e_{j'}$
are incident with the same vertex in $N_i$ for $i=1, \ldots, u$,
where $j':=(r_2 n_1+j)\!\!\mod{\lambda N n_1}$.
So, $e'_j=e'_{j'}$ for all $j \in [y\lambda Nn_1-r_2n_1]$.
Therefore, each edge $e \in E(\mathcal{K}_{n_1,\ldots, n_u})$
appears $x d(e)$ times in the sequence $S$, as $x r_2 = y \lambda N$.
On the other hand, $e \in E(\mathcal{K}_{n_1, \ldots, n_u})$ appears
$\lambda N'$ times in the sequence $S_{\lambda}(\mathcal{H})$,
as $\ell$ is an ordering of $\lambda\mathcal{K}_{n_1,\ldots, n_k}$.
Thus, $e$~appears $y \lambda N'$ times in the sequence~$S$.
Therefore,
$x d(e) = y \lambda N'$ for all $e \in E(\mathcal{K}_{n_1,\ldots , n_u})$,
and $d(e)$ is therefore constant.
By (\ref{eqn: rn sum}),
$d(e) n_1^u = r_2 n_1$;
hence, $n_1^{u-1} \mid r_2$.
\end{proof}

\section{Proof of Theorem~\ref{thm: hypergraph gen r}: Part II}\label{sec: Proof of the inequalities}

The next lemma required for the proof of Theorem~\ref{thm: hypergraph gen r}
is Lemma~\ref{lem: sequity gen case} below.
Before presenting this lemma, however,
let us first introduce notation used in this section and the next.

Recall that the representation of an integer $x$ in base $m$
is $x = (x_l, \ldots, x_0)_m$,
where $x = \sum_{i=0}^{l} x_i m^i$
and $x_i \in \mathbb{Z}_{m}$ for all~$i$.
We consider the following generalisation of this representation.
Let $m_1, \ldots , m_{k}$ be arbitrary positive integers
and set $M := \prod_{i=2}^{k}m_i$.
The representation of each integer $x \in \mathbb{Z}_{M}$ in base
$\overline{m} := (m_1, \ldots, m_k)$
is the
$k$-vector
$\langle x \rangle_{\overline{m}} := (0,x_{2}, \ldots , x_{k}) \in \{0\} \times \prod_{i=2}^{k}\mathbb{Z}_{m_i}$
that satisfies
\begin{equation}\label{eqn: rep of int}
x = \sum_{i=2}^{k} \Biggl( x_{i} \prod_{j=i+1}^{k} m_j \Biggr)\,.
\end{equation}
By the following lemma,
this representation is indeed well defined.
Note that the $0$ in the first coordinate is technically useful
as it will align with notation used later in the paper.

\begin{lemma}\label{lem: rep of int prop}
The representation $\langle x \rangle_{\overline{m}} := (0,x_{2}, \ldots, x_{k})$
of each $x\in \mathbb{Z}_M$ exists and is unique.
Furthermore,
$\langle x+1 \rangle_{\overline{m}} = (0,x_{2},\ldots ,x_{t-1} , x_{t}+1 \ldots , x_k+1)_{\overline{m}}$
for some $2 \leq t \leq k$.
\end{lemma}
\begin{proof}
Let $x \in \mathbb{Z}_M$ be an integer with representation
$\langle x \rangle_{\overline{m}} = (0,x_{2},\ldots , x_k)$.
Clearly, $x_k \equiv x \pmod{m_k}$.
Suppose, by induction, that $x_{l+1}, \ldots, x_k$ are uniquely determined by $x$.
Then,
as $x\equiv \sum_{i=l}^{k}x_i\prod_{j=i+1}^{k} m_j\pmod{\prod_{i=l}^{k} m_i}$
for any $2 \leq l \leq k$,
we can determine $x_l$ uniquely given $x$ and $x_{l+1}, \ldots, x_k$.
Thus, if an integer in $\mathbb{Z}_M$ has a representation $\langle x \rangle_{\overline{m}}$,
then it is unique.
As there are $M$ $k$-tuples,
each of which represents an integer satisfying (\ref{eqn: rep of int}),
every integer in $\mathbb{Z}_M$ has a unique representation as a $k$-tuple.

If $x_k \neq m_k-1$,
then clearly $ \langle x+1 \rangle = (0,x_{2},\ldots ,x_{k-1}, x_k+1)$, as required.
Otherwise,
let $t'$ be the smallest positive integer such that
$x_j = m_j-1$ for all $ t' < j \leq  k$.
Then
$x = \sum_{i=2}^{t'}(x_i \prod_{j=i+1}^{k}m_j) +
     \sum_{i=t'+1}^{k}\bigl((m_i-1) \prod_{j=i+1}^{k}m_j\bigr)$,
and so
\begin{align*}
x+1
&=
\sum_{i=2}^{t'} \Biggl( x_i \prod_{j=i+1}^{k}m_j \Biggr)+
\sum_{i=t'+1}^{k} \prod_{j=i}^{k}m_j -
\sum_{i=t'+1}^{k} \Biggl( \prod_{j=i+1}^{k}m_j \Biggr) + 1 \\
&=
\sum_{i=2}^{t'}\Biggl(x_i \prod_{j=i+1}^{k}m_j \Biggr) + \prod_{j=t'+1}^{k}m_j \,.
\end{align*}
Hence,
$x+1 = M$ if $t' = 1$,
and, if $t' \geq 2$, then
\[
  x+1 = \sum_{i=2}^{t'-1}\Biggl(x_i \prod_{j=i+1}^{k}m_j \Biggr) +(x_t'+1) \prod_{j=t'+1}^{k}m_j\,.
\]
Thus,
$\langle x+1 \rangle = \langle M \rangle_{\overline{m}} = \langle 0 \rangle_{\overline{m}} =(0,x_{2}+1, \ldots, x_k+1)$
when $t'=1$
and, when $t' \geq 2$,
$\langle x+1 \rangle_{\overline{m}}
= (0,x_{2}, \ldots, x_{t'-1},x_{t'}+1,\ldots, x_k+1)$.
In particular,
$\langle x+1 \rangle_{\overline{m}} = (0,x_{2}, \ldots, x_{t-1},x_{t}+1,\ldots, x_k+1)$
for some $t$,
namely $t = t'$ if $t' \geq 2$, and $t = 2$ if $t'=1$.
\end{proof}

\begin{lemma}\label{lem: sequity gen case}
For all $1 \leq n_1 \leq n_2 \leq \cdots \leq n_k$ and $r,\lambda \geq 1$,
\[
rn_1-1 \leq cms_r(\lambda \mathcal{K}_{n_1,\ldots,n_k})
\leq ms_r(\lambda \mathcal{K}_{n_1,\ldots,n_k}) \leq rn_1 \,.
\]
\end{lemma}
\noindent
To prove Lemma~\ref{lem: sequity gen case},
we need only consider cases, according to the following claim.
\begin{claim}\label{claim: reduction of cases gen case}
If Lemma~\ref{lem: sequity gen case} is true
for all $1 \leq r < N$ and $\lambda =1$,
then Lemma~\ref{lem: sequity gen case}
is true for all $r,\lambda  \geq 1$.
\end{claim}
\begin{proof}
\noindent
Suppose that Lemma~\ref{lem: sequity gen case} is true for all $r < N$ and $\lambda =1$.
Write $r$ as $r = r_1 N +r_2$
Then,
\[
  cms_r(\mathcal{H}) \geq r_1 n_1 N +
  cms_{r_2}(\mathcal{H}) \geq r_1 n_1 N + r_2 n_1 +1 = rn_1-1\,,
\]
by Lemma~\ref{lem: seqity gen r}.
Thus, by Lemma~\ref{lem: multi from non-multi},
$cms_r(\lambda \mathcal{H}) \geq r n_1 -1$ for all $\lambda \geq 1$,
and we can conclude that
$rn_1 - 1 \leq cms_r(\lambda \mathcal{K}_{n_1,\ldots,n_k})
          \leq  ms_r(\lambda \mathcal{K}_{n_1,\ldots,n_k}) \leq rn_1$,
as the two upper inequalities are trivially true.
\end{proof}
\noindent
To prove Lemma~\ref{lem: sequity gen case},
it therefore suffices to consider $\mathcal{K}_{n_1,\ldots,n_k}$.
More notation is however needed, so
let $d$ be a positive factor of $N$,
and let
$m_1,\ldots , m_k$ be integers satisfying $d = \prod_{i=2}^{k}m_i$
where $m_1=n_1$ and $m_i \mid n_i$ for all $2 \leq i \leq~k$.
Define $N_i := \mathbb{Z}_{m_i}\times \mathbb{Z}_{n_i/m_i}$ for $1 \leq i \leq k$,
$\overline{m} := (m_1, \ldots, m_k)$ and
$\overline{n/m} := (n_1/m_1, \ldots, n_k/m_k)$.
Without loss of generality, we can identify the edges of $\mathcal{K}_{n_1,\ldots,n_k}$
with the elements of $\mathcal{N}:= \prod_{i=1}^{k} N_i$;
in particular, each edge of $\mathcal{K}_{n_1,\ldots,n_k}$ is identified
with a vector
$\bigl((x_1,y_1),\ldots, (x_k,y_k)\bigr)_{\overline{m},\overline{n/m}}$.
The sum of two elements $((x_1,y_1),\ldots, (x_k,y_k))$, $((x'_1,y'_1),\ldots, (x'_k,y'_k)) \in \mathcal{N}$
is defined as $\bigl((x_1+x_1', y_{1}+y_{1}'), \ldots , (x_k+x_k', y_{k}+y_{k}')\bigr)_{\overline{m},\overline{n/m}}$.
The difference of two such elements is defined analogously.

For integers $x \in \mathbb{Z}_d$ and $y \in \mathbb{Z}_{N/d}$,
define $\langle (x,y) \rangle_{\overline{m},\overline{n/m}}
:= ((0,0),(x_2,y_2) \ldots ,\\(x_k,y_k))_{\overline{m},\overline{n/m}}$,
where
$\langle x \rangle_{\overline{m}}   = (0 ,x_2, \ldots, x_k)_{\overline{m}}$ and
$\langle y \rangle_{\overline{n/m}} = (0 ,y_2, \ldots, y_k)_{\overline{n/m}}$.
Also, for each integer $x \in [n_1]$,
define
$\langle x^* \rangle_{\overline{m},\overline{n/m}}
:= \bigl( (x_{1,1},x_{1,2}), \ldots , (x_{k,1},x_{k,2})\bigr)_{\overline{m},\overline{n/m}}$,
where
$x_{i,1} \in [m_i]$ and $x_{i,2} \in [\frac{n_i}{m_i}]$ satisfy
$x = x_{i,1}\frac{n_i}{m_i} + x_{i,2}$ for all $1 \leq i \leq k$.
It is easily checked that each $x_{i,j}$ is uniquely determined by~$x$,
and so $\langle x^* \rangle_{\overline{m},\overline{n/m}}$ is well defined.
Note that the first entry of $\langle x^* \rangle_{\overline{m},\overline{n/m}}$
is not necessarily equal to $(0,0)$.
The subscript $\overline{m},\overline{n/m}$ will be omitted if the context is implicitly clear.

For $i \in [d]$ and $j \in [\frac{N}{d}]$,
define
$
  \mathcal{M}_{i,j} := \bigl\lbrace \langle x^* \rangle_{\overline{m},\overline{n/m}} + \langle (i,j) \rangle_{\overline{m},\overline{n/m}} \;:\; x \in [n_1] \bigr\rbrace \,.
$
\begin{claim}\label{claim: matching decomp gen case}
The set $\bigl\lbrace \mathcal{M}_{i,j}\;:\; i \in [d], j \in [\frac{N}{d}] \bigr\rbrace$ is a matching
decomposition of $\mathcal{K}_{n_1,\ldots, n_k}$.
\end{claim}
\begin{proof}
We first check that each $\mathcal{M}_{i,j}$ is a matching.
Let
$\langle x^* \rangle = \bigl((x_{1,1},x_{1,2}), \ldots, \\(x_{k,1},x_{k,2})\bigr)$,
$\langle y^* \rangle = \bigl((y_{1,1},y_{1,2}), \ldots, (y_{k,1},y_{k,2})\bigr)$
and
$\langle (i,j) \rangle = \bigl((i_{1},j_1), \ldots, (i_k,j_k)\bigr)$
for distinct $x,y \in [n_1]$.
Suppose that the edges
$\langle x^* \rangle +\langle (i,j) \rangle$ and
$\langle y^* \rangle +\langle (i,j) \rangle$ in $\mathcal{M}_{i,j}$
have the same $l$-th entry for some $1 \leq l \leq k$; i.e.,
\[
x_{l,1}+i_l \equiv  y_{l,1} + i_l \pmod{m_l}
\qquad \textrm{and} \qquad
x_{l,2}+j_l \equiv
 y_{l,2}+ j_l \text{ }\Bigl(\text{mod }\frac{n_l}{m_l}\Bigr)\,.
 \]
Then $x_{l,1} = y_{l,1}$ and $x_{l,2} = y_{l,2}$.
Hence, $x = x_{l,1}\frac{n_l}{m_l}+x_{l,2} = y_{l,1}\frac{n_l}{m_l}+y_{l,2} = y$,
a contradiction.
Thus, $\mathcal{M}_{i,j}$ is a matching for all $i \in [d], j \in [\frac{N}{d}]$.

We now verify that the matchings $\mathcal{M}_{i,j}$ for $i \in [d], j \in [\frac{N}{d}]$
partition $E(\mathcal{K}_{n_1,\ldots , n_k})$.
As there are clearly $N$ matchings $\mathcal{M}_{i,j}$,
each containing $n_1$ edges, we need only show that no two distinct
$\mathcal{M}_{i,j}$ and $\mathcal{M}_{i',j'}$ contain a common edge.
Suppose, otherwise, that there are distinct $(i,j), (i',j')$ such that
$\mathcal{M}_{i,j}$ and $\mathcal{M}_{i',j'}$ contain a common edge.
By considering first entries,
it is easy to check that if $\mathcal{M}_{i,j}$ and $\mathcal{M}_{i',j'}$ contain a common edge,
then that edge is of the form
$\langle x^* \rangle +\langle (i,j) \rangle=
\langle x^* \rangle +\langle (i',j') \rangle$
for some $x \in [n_1]$.
Let $\langle x^* \rangle = ( (x_{1,1},x_{1,2}), \ldots , (x_{k,1},x_{k,2}))$,
$\langle (i,j) \rangle = ( (i_{1},j_1), \ldots , (i_k,j_k))$ and
$\langle (i',j') \rangle = ( (i'_{1},j'_1), \ldots , (i'_k,j'_k))$.
Then, by equating the $l$-th entries of  $\langle x^* \rangle +\langle (i,j) \rangle$ and
$\langle x^* \rangle +\langle (i',j') \rangle$, we see that, for $1 \leq l \leq k$,
\[
x_{l,1}+i_l \equiv x_{l,1}+i'_l
\pmod{m_i}
\quad \text{ and } \quad
x_{l,2}+ j_l \equiv x_{l,2}+ j_{l}'
\text{ }\Bigl(\text{mod } \frac{n_i}{m_i}\Bigr)\,.
\]
Then $(i_l,j_l) = (i_l',j_l')$ for all $1 \leq l \leq k$,
and so
\[
  \langle (i,j) \rangle = ( (i_{1},j_1), \ldots , (i_k,j_k))
 = ( (i'_{1},j'_1), \ldots , (i'_k,j'_k)) = \langle (i',j') \rangle\,,
\]
contradicting our assumption that $(i,j) \neq (i',j')$.
Hence, the matchings $\mathcal{M}_{i,j}$ for $i \in [d], j \in [\frac{N}{d}]$
are disjoint and, by the number of their edges, partition
$E(\mathcal{K}_{n_1,\ldots , n_k})$.
\end{proof}
Let $\ell_{i,j}$ be the ordering
of $\mathcal{M}_{i,j}$ defined by
$\ell_{i,j}\left(\langle x^* \rangle + \langle (i,j) \rangle \right) = x$ for all $x \in [n_1]$,
and set $\ell_{i,\frac{N}{d}}:= \ell_{i,0}$ and (thus)
$\mathcal{M}_{i,\frac{N}{d}}:= \mathcal{M}_{i,0}$ for all $i \in [d]$.
\begin{lemma}\label{lem: cyclic $k$-partite gcd not 1}
For all
$i \in [d]$ and $j \in [\frac{N}{d}]$,
$ms(\ell_{i,j},\ell_{i,j+1}) \geq n_1-1$.
\end{lemma}

\begin{proof}
Let $\ell = \ell_{i,j} \vee_{n_1-1} \ell_{i,j}$.
Consider a sequence $S$ of $n_1-1$ consecutive edges in $\ell$.
We check that $\mathcal{H}(S)$ is a matching of $\mathcal{K}_{n_1,\ldots, n_k}$.
Let $1 \leq s \leq n_1-2$ be the number of edges in~$S$ which are from $\mathcal{M}_{i,j}$.
There are then $n_1-1-s$ edges in $S$ from $\mathcal{M}_{i,j +1}$,
and the edges in $S$ which are from $\mathcal{M}_{i,j}$ are
$\langle x^*\rangle+\langle (i,j) \rangle$
for $n_1-s \leq a \leq n_1-1$,
and the edges in $S$ from $\mathcal{M}_{i,j+1}$ are
$\langle y^*\rangle+\langle (i,j+1)\rangle$
for $ 0 \leq y \leq n_1-s-2$.
As $\mathcal{M}_{i,j}$ and $\mathcal{M}_{i,j +1}$ are each matchings,
$\mathcal{H}(S)$ is not a matching only if there is an edge
from $\mathcal{M}_{i,j}$ in $S$ and another from $\mathcal{M}_{i,j+1}$ in $S$
that have a common entry.
So, suppose that $\langle x^*\rangle+\langle (i,j)\rangle$
and $\langle y^*\rangle+\langle (i,j+1) \rangle$
have the same $l$-th entry for some $1 \leq l \leq k$, $n_1-s \leq x \leq n_1-1$ and
$0 \leq y \leq n_1-s-2$.
Let $\langle x^* \rangle = \bigl((x_{1,1},x_{1,2}),\ldots, (x_{k,1},x_{k,2})\bigr)$,
$\langle y^* \rangle = \bigl((y_{1,1},y_{1,2}), \ldots, (y_{k,1},y_{k,2})\bigr)$
and $\langle (i,j) \rangle = \bigl((i_{1},j_1), \ldots, (i_k,j_k)\bigr)$.
By Lemma~\ref{lem: rep of int prop},
$\langle (i,j+1) \rangle
= \bigl((i_{1},j_1), \ldots, (i_{t-1},j_{t-1}),(i_t,j_{t}+1), \ldots, (i_k,j_k+1)\bigr)$
for some $2 \leq t \leq k$.
Thus, by equating the $\ell$th entries of $\langle x^*\rangle+\langle (i,j)\rangle$
and $\langle y^*\rangle+\langle (i,j+1) \rangle$,
we see that
\begin{equation}\label{eqn:entries gen cyclic case}
(x_{\ell,1}+i_l,x_{\ell,2}+j_\ell) =
\begin{cases}
\bigl((y_{\ell,1}+i_{\ell}) \!\!\!\mod{m_\ell}, (y_{\ell,2}+j_{\ell})  \!\!\!\mod{\frac{n_l}{m_l}}\bigr) &\text{ if } l < t \\
\bigl((y_{\ell,1}+i_{\ell}) \!\!\!\mod{m_\ell}, (y_{\ell,2}+j_{\ell}+1)\!\!\!\mod{\frac{n_l}{m_l}}\bigr) & \text{ otherwise.}
\end{cases}
\end{equation}
By equating the entries of the pairs in (\ref{eqn:entries gen cyclic case}),
we see that $x_{l,1} = y_{l,1}$ and either $x_{l,2} = y_{l,2}$
or $x_{l,2} \equiv y_{l,2}+1 \pmod{\frac{n_l}{m_l}}$.
If the former is true, then
$x = x_{l,1}\frac{n_l}{m_l} + x_{l,2} = y_{l,1}\frac{n_l}{m_l} + y_{l,2} =~y$,
a contradiction.
Hence, $x_{l,2} \equiv y_{l,2}+1 \pmod{\frac{n_l}{m_l}}$.
If $x_{l,2} = y_{l,2}+1$, then, using a similar argument,
we arrive at the contradiction $x = y+1$.
We are then left with the case in which
$x_{l,2} = 0$ and $y_{l,2} = \frac{n_l}{m_l}-1$,
also a contradiction, as, otherwise,
$x = x_{l,1}\frac{n_l}{m_l} < y_{l,1}\frac{n_l}{m_l} + \frac{n_{l}}{m_l}-1 = y$.
Hence, $\mathcal{H}(S)$ is a matching, as required.
\end{proof}
\noindent
We can now prove Lemma~\ref{lem: sequity gen case}.
\begin{proof}[Proof of Lemma~\ref{lem: sequity gen case}]
Let $r < N$ and $d = \gcd(N,r)$.
By Claim~\ref{claim: matching decomp gen case} and
Lemma~\ref{lem: cyclic $k$-partite gcd not 1},
the assumptions of Proposition~\ref{prop: Matching decomposition} are met
for the hypergraph $\mathcal{K}_{n_1,\ldots , n_k}$ with
matchings $\mathcal{M}_{i ,j}$ ordered by
$\ell_{i ,j}$ for $i \in [d]$ and $j \in [\frac{N}{d}] $, respectively.
Thus,
$cms_r(\mathcal{K}_{n_1,\ldots , n_k}) \geq rn_1-1$ when $r<N$.
By Claim~\ref{claim: reduction of cases gen case}, Lemma~\ref{lem: sequity gen case} is
true for all $r \geq 1$ and $\lambda \geq 1$.
\end{proof}

\section{Proof of Theorem~\ref{thm: hypergraph gen r}: Part III}
\label{sec: remaining}
We now present the remaining lemmas required for the proof of Theorem~\ref{thm: hypergraph gen r},
namely, Lemmas~\ref{lem: hypergraph special cyclic case} and~\ref{lem: non-cyclic case simplified}.
\begin{lemma}\label{lem: hypergraph special cyclic case}
If $n_1^{u-1} \mid r_2$, then $cms_r(\lambda \mathcal{K}_{n_1,\ldots,n_k}) = rn_1$.
\end{lemma}
\begin{lemma} \label{lem: non-cyclic case simplified}
If $n_1^{u-1} \mid r_2$ or
\begin{equation}\label{eqn: non cyclic assume condition}
\left(\left\lfloor \frac{r_2}{n_1^{u-1}}\right\rfloor +1 \right)
\left\lfloor \frac{\lambda}{r_2}\prod_{i=2}^{k} n_i \right\rfloor
\leq \lambda\prod_{i=u+1}^{k}n_i \leq
\left\lfloor  \frac{r_2}{n_1^{u-1}}\right\rfloor
\left(\left\lfloor \frac{\lambda}{r_2}\prod_{i=2}^{k} n_i \right\rfloor +1 \right)  \,,
\end{equation}
then $ms_r(\mathcal{K}_{n_1,\ldots , n_k}) = r n_1$.
\end{lemma}
\noindent
The rest of this section serves to prove these lemmas.

First note that we can immediately reduce Lemma~\ref{lem: hypergraph special cyclic case} to
a single case for $\lambda$, as follows.
\begin{claim}\label{claim: reduction cyclic case}
If Lemma~\ref{lem: hypergraph special cyclic case} is true for $r=n_1^{u-1}$
and $\lambda =1$,
then Lemma~\ref{lem: hypergraph special cyclic case} is true for all $r \geq n_1^{u-1}$ and $\lambda \geq 1$.
\end{claim}
\begin{proof}
Suppose that Lemma~\ref{lem: hypergraph special cyclic case} is true
for $r=n_1^{u-1}$ and $\lambda = 1$.
Then, by Lemma~\ref{lem: 1 to r},
Lemma~\ref{lem: hypergraph special cyclic case} is true for $\lambda=1$ and all $r < N$
such that
$n_1^{u-1} \mid r_2=r$.
Thus, for any $r =r_1N +r_2$ such that
$n_1^{u-1} \mid r_2$, it follows from Lemma~\ref{lem: seqity gen r} that
\[
cms_r(\mathcal{K}_{n_1,\ldots,n_k}) \geq r_1 n_1 N + cms_{r_2}(\mathcal{H}) = r_1 n_1N +r_2 n_1 = r n_1 \,.
\]
By Lemma~\ref{lem: multi from non-multi}, the cases in which $\lambda >1$ follow from
the case in which $\lambda=1$, and we are done.
\end{proof}

\begin{claim}\label{claim: reduction non-cyclic case}
If Lemma~\ref{lem: non-cyclic case simplified} is true for $1 \leq r < \lambda N$,
then Lemma~\ref{lem: non-cyclic case simplified} is true for all $r \geq 1$.
\end{claim}
\begin{proof}
Suppose that Lemma~\ref{lem: non-cyclic case simplified} is true for $r < \lambda N$
and that either $n_1^{u-1}\mid r_2$ or equation (\ref{eqn: non cyclic assume condition}) holds.
Then, for each $r\geq 1$,
\[
ms_r(\lambda\mathcal{K}_{n_1,\ldots,n_k})
\geq r_1 n_1 \lambda \prod_{i=2}^k n_i + ms_{r_2}(\lambda\mathcal{K}_{n_1,\ldots,n_k})
= r_1 \lambda n_1 \prod_{i=2}^k n_i+r_2 n_1 = r n_1 \,,
\]
by Lemma~\ref{lem: seqity gen r}.
\end{proof}

Set $N_i := [n_i]$ for all $1 \leq i \leq k$.
By the natural isomorphism between $[n_i]$ and $[n_i] \times [1]$ for all~$i$,
it follows that the sets $N_i$ are, up to isomorphism,
the same sets as those defined in Section~\ref{sec: Proof of the inequalities}
for $d=N$; i.e., when $m_i = n_i$ for all $1 \leq i \leq k$.
We will therefore use the definitions and notation of the previous section,
where, for simplicity,
we identify the edges of $\mathcal{K}_{n_1,\ldots,n_k}$ with the elements of
$\prod_{i=1}^{k} \mathbb{Z}_{n_i}$.
Then $\langle x^* \rangle_{\overline{n}}:=
\langle x^* \rangle_{\overline{m},\overline{n/m} }$, as defined in
Section~\ref{sec: Proof of the inequalities},
will be identified  with the element $(x, \ldots , x) \in \prod_{i=1}^{k}\mathbb{Z}_{n_i}$
for each $x\in \mathbb{Z}_{n_1}$.

Let $\ell'$ be a labelling of $\mathcal{K}_{n_1, \ldots, n_u}$
such that the edges
$(\ell')^{-1}(x n_1), \ldots, (\ell')^{-1}(x n_1 + n_1-1)$
form a matching for all $x \in [n_1^{u-1}]$.
That is, let $\ell'$ be an ordering which corresponds to
$S(\mathcal{M}_0) \vee \cdots \vee S(\mathcal{M}_{n_1^{u-1}-1})$
for some matching decomposition
$\mathcal{M}_0$, $\ldots$, $\mathcal{M}_{n_1^{u-1}-1}$ of
$\mathcal{K}_{n_1,\ldots , n_u}$, where each $\mathcal{M}_i$ is ordered arbitrarily.
Let $\overline{n}' := (n_1,n_{u+1},\ldots, n_{k})$.
For $i \in [N']$,
let
$\mathcal{M}'_{i} := \bigl\lbrace \langle x^* \rangle_{\overline{n}'} - \langle i \rangle_{\overline{n}'}
\;:\; x \in [n_1] \bigr\rbrace$
and
$\overline{\mathcal{M}'_{i}} := \bigl\lbrace (x_{u+1}, \ldots , x_k)\;:\;
(x_1 ,x_{u+1}, \ldots, x_{k})\in \mathcal{M}_i' \bigr\rbrace$.
It is easy to check that $\mathcal{M}'_i$ and, therefore,
$\overline{\mathcal{M}'_{i}}$ is a matching,
by using a similar argument to the proof of Claim~\ref{claim: matching decomp gen case}.
Let $\ell'_{i}$ be the ordering of $\mathcal{M}'_i$ defined by
$\ell'_i(\langle x^* \rangle_{\overline{n}'}-\langle i \rangle_{\overline{n}'}) = x$
for all $x \in [n_1]$.
Also let $\overline{\ell_i'}$ be the analogous ordering for $\overline{\mathcal{M}_i'}$.
Identify each element $(x_1 , \ldots ,x_k) \in \prod_{i=1}^{k} \mathbb{Z}_{n_i}$
with element $\bigl((x_1 , \ldots, x_u), (x_{u+1},\ldots, x_{k})\bigr) \in
(\prod_{i=1}^{u} \mathbb{Z}_{n_i}) \times (\prod_{i=u+1}^{k}\mathbb{Z}_{n_i})$.
For $i \in [n_1^{u-1}]$ and $j \in [\lambda N']$,
let $\mathcal{M}_{i,j}'$ be a set containing
an edge parallel to the edge $((\ell')^{-1}(i n_1 + x),(\overline{\ell_{j}'})^{-1}(x))$
for each $x \in [n_1]$, where, for simplicity, we let $\ell'_{j}= \ell'_{j'}$ for $j' \in [N']$
with $j' \equiv j \pmod{N}$.

\begin{claim}\label{claim: decomp for N'}
The set $\bigl\lbrace \mathcal{M}_{i,j}' \;:\; i \in [n_1^{u-1}], j \in [\lambda N']\bigr\rbrace$
is a matching decomposition of $\lambda \mathcal{K}_{n_1 ,\ldots , n_k}$.
\end{claim}
\begin{proof}
Each $\mathcal{M}_{i,j}'$ is a matching since $(\ell')^{-1}(in_1) , \ldots, (\ell')^{-1}(in_1+n_1-1)$
form the matching $\mathcal{M}_{i}$
and $(\ell_{j}')^{-1}(0) , \ldots, (\ell_{j}')^{-1}(n_1-1)$
form the matching $\overline{\mathcal{M}_i}$.
For $i~\in~[n_1^{u-1}]$ and $j \in [N']$,
there are $\lambda$ matchings
whose edges are parallel to the same as those in $\mathcal{M}_{i,j}'$,
namely, $\mathcal{M}_{i,j}', \ldots, \mathcal{M}'_{i,j+(\lambda-1)N'}$.
Therefore, it suffices to show that
$\bigl\lbrace \mathcal{M}_{i,j} \;:\; i \in [n_1^{u-1}], j \in [N'] \bigr\rbrace$ is a matching decomposition of $\mathcal{K}_{n_1 ,\ldots , n_k}$.

We see that the matching $\mathcal{M}'_{i}$ is isomorphic to the matching
$\mathcal{M}_{0,N'-i}$ defined in
Section~\ref{sec: Proof of the inequalities} for $d=N'$ and
$\mathcal{K}_{n_1,n_{u+1},\ldots, n_k}$,
by noting that $ \langle x^* \rangle_{\overline{n}'} + \langle -i \rangle_{\overline{n}'}
= \langle x^* \rangle_{\overline{n}'} - \langle i \rangle_{\overline{n}'}$
for any $x \in [n_1]$
and by setting $\mathcal{M}_{0,N'}:= \mathcal{M}_{0,0}$.
Thus, by Claim~\ref{claim: matching decomp gen case},
$\bigl\lbrace \mathcal{M}'_{j}\;:\; j \in [N'] \bigr\rbrace$ is a matching
decomposition of $\mathcal{K}_{n_1, n_{u+1},\ldots , n_k}$.
The edges of $\mathcal{M}_{i,j}'$ are isomorphic to edges in
$\mathcal{M}'_j$ by identifying
$((\ell')^{-1}(i n_1 + x),(\overline{\ell_{j}'})^{-1}(x))$ with
$(x,(\overline{\ell_{j}'})^{-1}(x))$
for all $x \in [n_1]$.
Hence,
$\bigl\lbrace \mathcal{M}_{i,j}'\;:\; j \in [N'] \bigr\rbrace$ is a matching
decomposition of $\mathcal{M}_i \times \mathcal{K}_{ n_{u+1},\ldots , n_k}$
for any $i \in [ n_1^{u-1}]$.
As every edge of $\mathcal{K}_{n_1, \ldots, n_{u}}$
appears in exactly one $\mathcal{M}_i$, 
the set $\bigl\lbrace \mathcal{M}_{i,j}'\;:\; i \in [n_1^{u-1}],\, j \in [N'] \bigr\rbrace$
is a matching decomposition of $\mathcal{K}_{ n_{1},\ldots , n_k}$, as required.
\end{proof}

Let $\ell_{i,j}'$ be the ordering of $\mathcal{M}_{i,j}'$ defined by
$\ell_{i,j}'\left(\bigl(({\ell'})^{-1}(i n_1 + x),(\overline{\ell_{j}'})^{-1}(x)\bigr)\right)$
for $x \in [n_1]$,
and set $\ell'_{i,\lambda N}:= \ell'_{i,0}$ and (thus) $\mathcal{M}'_{i,\lambda N}:= \mathcal{M}'_{i,0}$.

\begin{lemma}\label{lem: seqity hypergraph one min}
For all $ i \in [n_1^{u-1}]$ and  $j \in [\lambda N']$, $ms(\ell'_{i,j},\ell'_{i,j+1}) \geq n_1$ holds.
\end{lemma}
\begin{proof}
Let $\ell = \ell_{i,j}'\vee_{n_1}\ell'_{i,j+1}$.
Consider a sequence $S$ of $n_1$ consecutive edges
in $\ell$.
The edges of $S$ that appear in the matching $\mathcal{M}_{i,j}$ (in order with respect to $\ell$) are
\[
  \bigl((\ell')^{-1}(i n_1 + x),(\overline{\ell_{j}'})^{-1}(x)\bigr), \ldots,
  \bigl((\ell')^{-1}(i n_1 + n_1-1),(\overline{\ell_{j}'})^{-1}(n_1-1)\bigr)
\]
and the edges of $S$ that appear in the matching $\mathcal{M}_{i,j+1}$ (in order with respect to $\ell$) are
\[
  \bigl((\ell')^{-1}(i n_1),(\overline{\ell_{j+1}'})^{-1}(0)\bigr), \ldots ,
  \bigl((\ell')^{-1}(i n_1 + x-1),(\overline{\ell_{j+1}'})^{-1}(x-1)\bigr)
\]
for some $1 \leq x \leq n_1-1$.
The edges $(\ell')^{-1}(i n_1 + 0), \ldots, (\ell')^{-1}(i n_1 + n_1-1)$ form
the matching $\mathcal{M}_i$ and, in particular, every vertex in $[n_l]$ for $1 \leq l \leq u$
has degree $1$ in $\mathcal{H}(S)$.

So without loss of generality, we consider the degree of vertices in $[n_l]$ for $u+1 \leq l \leq k$
in the hypergraph $\mathcal{H}(S')$, where $S' = (\overline{\ell_{j}'})^{-1}(x), \ldots, (\overline{\ell_{j}'})^{-1}(n_1-1), \\
(\overline{\ell_{j+1}'})^{-1}(0), \ldots, (\overline{\ell_{j+1}'})^{-1}(x-1)$.
As we are not concerned with the degree of vertices in $[n_1]$, we can consider the
hypergraph formed by the edges
\begin{equation}\label{eqn: edges of N' bit}
\begin{split}
\langle x^* \rangle_{\overline{n}'}-\langle j \rangle_{\overline{n}'}, &\ldots ,
\langle (n_1-1)^* \rangle_{\overline{n}'}-\langle j \rangle_{\overline{n}'}, \\
\langle 0^* \rangle_{\overline{n}'}-\langle j+1 \rangle_{\overline{n}'}, &\ldots ,
\langle (x-1)^* \rangle_{\overline{n}'}-\langle j+1 \rangle_{\overline{n}'} \,,
\end{split}
\end{equation}
by ignoring the first entry of each edge.
Let $\langle j \rangle_{\overline{n}'} = (j_1 , j_{u+1}, \ldots, j_k)_{\overline{n}'}$.
By Lemma~\ref{lem: rep of int prop}
$\langle j+1 \rangle_{\overline{n}'} = (j_1 , j_{u+1}, \ldots, j_{u+t-1} , j_{u+t}+1, \ldots, j_k+1)_{\overline{n}'}$
for some $1 \leq t \leq k-u$.
For $u+1 \leq l \leq u+t-1$,
the $(l-u+1)$-th entry of the edges in (\ref{eqn: edges of N' bit})
are, modulo $n_l$, $x-j_{l} ,\ldots, n_1-1-j_l$ and $-j_l ,1 -j_l , \ldots, x-1-j_l$,
which are clearly distinct as $n_l > n_1$.
For $u+t \leq l \leq k$,
the $(l-u+1)$-th entry of the edges in (\ref{eqn: edges of N' bit}) modulo $n_l$
are $x-j_{l} ,\ldots, n_1-1-j_l$ and $-j_l-1 , -j_l , \ldots, x-2-j_l$,
which are distinct since $n_l > n_1$.
Thus, every vertex in $[n_l]$ for $u+1 \leq l \leq k$ is incident
with at most one edge in (\ref{eqn: edges of N' bit})
and thus at most one edge in $S$.
Hence, $\mathcal{H}(S)$ is a matching.
\end{proof}
\noindent
\begin{proof}[Proof of Lemma~\ref{lem: hypergraph special cyclic case} ]
By Claim~\ref{claim: reduction cyclic case}, we only need to consider the
case in which $r=n_1^{u-1}$ and $\lambda=1$.
By Claim~\ref{claim: decomp for N'},
$\bigl\lbrace \mathcal{M}'_{i,j}\;:\; i \in [n_1^{u-1}], j \in [N'] \bigr\rbrace$
is a matching decomposition of $\mathcal{K}_{n_1, \ldots , n_k}$.
By Lemma~\ref{lem: seqity hypergraph one min},
$ ms(\ell_{i,j}',\ell_{i,j+1}') \geq n_1$ for all $i \in [n_1^{u-1}]$ and $j \in [N']$.
Hence, by Proposition~\ref{prop: Matching decomposition},
we have that $cms_r(\mathcal{K}_{n_1, \ldots , n_k}) \geq r n_1$,
and so $cms_r(\mathcal{K}_{n_1, \ldots , n_k}) = r n_1$
as required.
\end{proof}

The remainder of this section is devoted to proving
Lemma~\ref{lem: non-cyclic case simplified}.
We assume that $n_1^{u-1} \nmid r_2$,
as the case in which $n_1^{u-1} \mid r_2$ has been shown in
Lemma~\ref{lem: hypergraph special cyclic case}.
Let $r < \lambda N$ be a positive integer and write $r = p n_1^{u-1}+q$
for non-negative integers $p$ and $q$ such that
$0 < q < n_1^{u-1}$,
and recall that $\lambda N = ar+b$.
Then (\ref{eqn: non cyclic assume condition})
can be expressed as
\begin{equation}\label{eqn: Bound on para}
(p+1)a \leq \lambda N' \leq  p(a+1)\,.
\end{equation}
As we are proving Lemma~\ref{lem: non-cyclic case simplified},
we will assume that (\ref{eqn: Bound on para}) holds and thus that $p \neq 0$.

Let
$\alpha = p$,
$\beta  = (p+1)$,
$\gamma = (\lambda N' - ap)$,
$\delta = \bigl(\lambda N' - a(p+1)\bigr)$ and
$ \nu = n_1^{u-1} - q$.
The~identities $r = p n_1^{u-1}+q$ and $n_1^{u-1}\lambda N'= ra + b$
easily yield the following expressions:
\begin{align}
                    \gamma \nu+ \delta q &= b \label{eqn: gamma + delta}\\
(\alpha - \gamma)\nu + (\beta - \delta)q &= r -b \label{eqn: r-b}\\
         a \alpha  + \gamma = \lambda N' &= a \beta  + \delta\,. \label{eqn:number of alpha and beta edges}
\end{align}
By (\ref{eqn: Bound on para}),
each of the numbers $\gamma, \delta , \alpha- \gamma$ and $\beta- \delta$
is non-negative.

Let $\sigma \;: \;[\lambda N'] \rightarrow [\lambda N']$ be a function with the properties given in
Corollary~\ref{cor: General non cyclic ordering} with $s=\alpha$ and $t = \lambda N'$.
Similarly, let $\tau \;: \;[\lambda N'] \rightarrow [\lambda N']$ be a function with the properties given in Corollary~\ref{cor: General non cyclic ordering} with $s=\beta$ and $t = \lambda N'$.
For a fixed pair $(i,j) \in [n_1^{u-1}] \times [\lambda N']$, let $s_{i,j}$ and $t_{i,j}$ be the integers
that satisfy
\[
\begin{cases}
\sigma(j) = s_{i,j} \alpha +t_{i,j}  \text{ with } t_{i,j} \in [\alpha] &
\text{ if } i \in [\nu]\,;   \\
\tau(j) = s_{i,j} \beta +t_{i,j}  \text{ with } t_{i,j} \in [\beta]
&  \text{ otherwise. }
\end{cases}
\]

Let $\rho \;:\; [n_1^{u-1}]\times [\lambda N'] \rightarrow [\lambda N]$ be defined by
\[
\rho(i,j) =
\begin{cases}
s_{i,j} r+\nu t_{i,j}+i &
\text{if } t_{i,j} \in [\gamma] \text{ and } i \in [\nu]\,;\\
s_{i,j} r+\nu \gamma+q t_{i,j} +i-\nu &
\text{if } t_{i,j} \in [\delta] \text{ and } i \in [\nu+q]-[\nu ]\,; \\
s_{i,j} r+b+\nu (t_{i,j}-\gamma)+i &
\text{if } t_{i,j} \in [\alpha]-[\gamma] \text{ and } i \in [\nu]\,;\\
s_{i,j} r+b+\nu(\alpha-\gamma)+qt_{i,j}+i-\nu &
\text{if } t_{i,j} \in [\beta]-[\delta] \text{ and } i \in [\nu+q]-[\nu].
\end{cases}
\]
As $\sigma$ and $\tau$ are bijections of $[\lambda N']$,
(\ref{eqn: Bound on para}) implies that $s_{i,j} \in [a+1]$ for all $i$ and $j$.
Furthermore by (\ref{eqn:number of alpha and beta edges}), if $s_{i,j}=a$,
then $t_{i,j} \in [\gamma]$ if $i \in [\nu]$ and $t \in [\delta]$ otherwise.
Therefore, if $i \in [\nu]$,
then either $\rho(i,j) = s_{i,j} r+\nu t_{i,j}+i \leq ar + \nu(\gamma-1)+\nu-1$
or $\rho(i,j) = s_{i,j} r+b+\nu (t_{i,j}-\gamma)+i \leq (a-1)r+b+\nu (\alpha-1-\gamma)+\nu-1$.
In either case, $\rho(i,j) < \lambda N$,
by (\ref{eqn: gamma + delta}) and (\ref{eqn: r-b}), respectively.
By a similar argument,
$\rho(i,j) < \lambda N$ when $i \in [\nu+q]-[\nu]$,
and $\rho$ is thus well defined.
\begin{lemma}\label{lem: properties of difficult re-index non-cyclic case}
The function $\rho$ is an ordering of $[n_1^{u-1}]\times [\lambda N']$
with the property that
if $\rho(i,j)\in[\lambda N-r]$,
then $\rho(i,j+1) = \rho(i,j)+r$.
\end{lemma}
\begin{proof}
We first check that $\rho$ is an ordering of $[n_1^{u-1}]\times [\lambda N']$.
Suppose that $\rho(i,j) = \rho(i',j')$.
By inspection, we have that
\[
\rho(i,j) -s_{i,j}r \in
\begin{cases}
[\nu \gamma] &
\text{ for } t_{i,j} \in [\gamma] \text{ and } i \in [\nu ]\,;\\
[b] -[ \nu \gamma] &
\text{ for } t_{i,j} \in [\delta] \text{ and } i \in [\nu +q]-[\nu]\,;\\
[b+\nu (\alpha-\gamma)]-[b]&
\text{ for } t_{i,j} \in [\alpha]-[\gamma] \text{ and } i \in [\nu]\,;\\
[r] - [b+\nu (\alpha-\gamma)]&
\text{ for } t_{i,j} \in [\beta]-[\delta] \text{ and } i \in [\nu+q]-[\nu]\;,
\end{cases}
\]
and $\rho(i',j')-s_{i',j'}r$ has the analogous property.
Therefore, $s_{i,j} = s_{i',j'}$ and
either $i ,i' \in [\nu]$ or $i,i' \in [\nu+q]-[\nu]$.
Thus, by the definition of $\rho$,
\begin{equation}\label{eqn:bij proof of diff ord}
0 = \rho(i,j)-\rho(i',j') =
\begin{cases}
\nu(t_{i,j}-t_{i',j'})+i-i' &\text{ if } i,i' \in [\nu ]\,;\\
q(t_{i,j}-t_{i',j'})+i-i' &\text{ if } i,i' \in [\nu+q ]-[\nu]\,.
\end{cases}
\end{equation}
However, $|i-i'| \in [\nu]$ if $i,i' \in [\nu]$,
 and $|i-i'| \in [q]$ if $i,i' \in [\nu+q ]-[\nu]$.
Hence, (\ref{eqn:bij proof of diff ord}) implies that $t_{i,j}=t_{i',j'}$ and $i=i'$.
Therefore, $j = \sigma^{-1}(s_{i,j}\alpha + t_{i,j}) = \sigma^{-1}(s_{i',j'}\alpha + t_{i',j'})=j'$
if $i\in [\nu]$,
and, similarly, $j = j'$ if $i \in [\nu +q]-[\nu]$.
Thus, $(i,j)=(i',j')$, and so $\rho$ is injective.
Since $|[n_1^{u-1}]\times [\lambda N']| = |[\lambda N]|$, $\rho$ is a bijection and,
hence, an ordering of $[n_1^{u-1}]\times [\lambda N']$.

We now check that $\rho$ satisfies the property given in the lemma.
Suppose that $\rho(i,j) \in [\lambda N-r]$.
If $i \in [\nu]$, then $s_{i,j}r + \nu t_{i,j}+i < (a-1)r+b$ if $t \in [\gamma]$,
and $s_{i,j}r+b + \nu (t_{i,j} -\gamma)+i < (a-1)r+b$ otherwise.
Therefore, $s_{i,j} \leq a-1$ and if $s_{i,j} = a-1$, then $t_{i,j} \in [\gamma]$.
Thus, $s_{i,j}\alpha +t_{i,j} < \lambda N' - \alpha$ and so,
by Corollary~\ref{cor: General non cyclic ordering},
$\sigma(j+1) = (s_{i,j}+1)\alpha +t_{i,j}$.
By a similar argument,
$\tau(j+1) = (s_{i,j}+1)\beta + t_{i,j}$ for $i \in [\nu + q]-[\nu]$.
Hence, in any case, $s_{i,j+1} = s_{i,j} +1 $ and $t_{i,j+1} = t_{i,j}$.
By the definition of $\rho$,
$\rho(i,j+1)-s_{i,j+1}r = \rho(i,j)-s_{i,j}r$,
and so $\rho(i,j+1)-\rho(i,j) = s_{i,j+1}r - s_{i,j}r = r$.
Rearranging yields the required expression.
\end{proof}
\noindent
\begin{proof}[Proof of Lemma~\ref{lem: non-cyclic case simplified}]
By Claim~\ref{claim: reduction non-cyclic case},
we only need to consider the cases in which $1 \leq r < \lambda N$.
Let $\mathcal{M}_l = \mathcal{M}'_{\rho^{-1}(l)}$
and $\ell_l = \ell'_{\rho^{-1}(l)}$ for all $l \in [\lambda N]$.
We check that the conditions of Proposition~\ref{prop: Matching decomposition both}
are satisfied for the matchings
$\mathcal{M}_{0}, \ldots, \mathcal{M}_{\lambda N-1}$
of $\mathcal{K}_{n_1, \ldots , n_k}$.
By Claim~\ref{claim: decomp for N'},
$\bigl\lbrace \mathcal{M}_{0}, \ldots, \mathcal{M}_{\lambda N-1}\bigr\rbrace$
is a matching decomposition of $\mathcal{K}_{n_1, \ldots , n_k}$.
For $l \in [\lambda N-r]$,
let $\rho^{-1}(l) = (i,j)$ and so $\rho(i,j) =l$.
By Lemma~\ref{lem: properties of difficult re-index non-cyclic case},
$\rho(i,j+1) = \rho(i,j)+r = l+r$
and, as $\rho$ is a bijection, $\rho^{-1}(l+r) = (i,j+1)$.
Hence, $ms(\ell_{l},\ell_{l+r}) = ms(\ell_{i,j},\ell_{i,j+1}) \geq n_1$,
by Lemma~\ref{lem: seqity hypergraph one min}, and the proof then follows
from Proposition~\ref{prop: Matching decomposition both}.
\end{proof}

\section{Proof of Theorem~\ref{thm: hypergraph gen r}: Conclusion}
\label{sec: proof conclusion}

\begin{proof}[Proof of Theorem~\ref{thm: hypergraph gen r}]\hspace{-1.21pt}By Lemma \ref{lem: sequity gen case},~$ms_r(\lambda \mathcal{K}_{n_1, \ldots , n_k})$~and~$cms_r(\lambda \mathcal{K}_{n_1, \ldots , n_k})$ are each
either $rn_1 - 1$ or $rn_1$.
By Lemmas~\ref{lem: hypergraph necessary condition}
and \ref{lem: non-cyclic case simplified},
$ms_r(\lambda \mathcal{K}_{n_1, \ldots , n_k}) = rn_1$
if and only if $n_1^{u-1} \mid r$ or
(\ref{eqn: cond hypergraph comp}) holds. Thus,
\[
ms_r(\lambda\mathcal{K}_{n_1,\ldots,n_k}) =
\begin{cases}
rn_1 &\text{ if } n_1^{u-1}\mid r_2
\text{ or } (\ref{eqn: cond hypergraph comp}) \text{ holds}\,; \\
rn_1-1 &\text{ otherwise}\,.
\end{cases}
\]
Similarly by Lemmas~\ref{lem: hypergraph cyclic necessary condition}
and \ref{lem: hypergraph special cyclic case}
$cms_r(\lambda \mathcal{K}_{n_1, \ldots , n_k}) = rn_1$
if and only if $n_1^{u-1} \mid r$ and, thus,
\[
\hspace*{-64pt} cms_r(\lambda\mathcal{K}_{n_1,\ldots,n_k}) =
\begin{cases}
rn_1 &\text{ if } n_1^{u-1}\mid r_2\,;
\\
rn_1-1 &\text{ otherwise}\,.
\end{cases}
\]
\end{proof}

\section{Concluding Remarks}\label{sec: Con}

One can show,
for the special case in which $p=1,q=0$, $\lambda=1$,
and where $\sigma$ is the identity function on $[\lambda N']$,
that the function $\rho$ defined in Section~\ref{sec: remaining}
reduces to the much simpler function
$\rho(i,j) = jr+i$ for all $i \in [n_1^{u-1}]$
and $j \in [N']$ and, furthermore, that it satisfies a cyclic analogue of
Lemma~\ref{lem: properties of difficult re-index non-cyclic case},
namely $\rho(i,j+1) = (\rho(i,j) + r) \,\,\,  \text{modulo} \, N$
for all $i \in [n_1^{u-1}]$ and $j \in [N']$.
The given proof of Lemma~\ref{lem: hypergraph special cyclic case}
implicitly uses this $\rho$:
Proposition~\ref{prop: Matching decomposition} uses Lemma~\ref{lem: General cyclic ordering}.
The cyclic construction in the previous section is thus a very special case of the
non-cyclic construction.

Though the hypergraphs in this paper attain the lower bounds
in Lemma~\ref{lem: seqity gen r}, there are hypergraphs
which do not. Consider the graph $G$ below.
\begin{figure}[h]
\begin{minipage}[b][5cm]{.44\textwidth}
\centering
\begin{tikzpicture}[thick,scale=.7175]
  \pgfmathsetlengthmacro\scfac{2.5cm}
  \pgfmathsetlengthmacro\scfacnew{1.41*\scfac}
  \pgfmathsetmacro{\sepang}{120}
  \pgfmathtruncatemacro{\c}{3}
\draw[line width = \scfac*0.02, color = blue]{

(0,0)-- node[pos = 0.5, above right] {0} (135:\scfacnew)
(0,0) -- node[pos = 0.5, below right] {2} (135+90:\scfacnew)
(0,0) -- node[pos = 0.5, below left] {4} (-45:\scfacnew)
(135:\scfacnew)--node[color = black, pos = 0.5, left] {$e$} (135+90:\scfacnew)
(45:0.5*\scfacnew)+(-45:0.5*\scfacnew)--node[pos = 0.5, right] {} (-45:\scfacnew)
(45:0.5*\scfacnew)+(-45:0.5*\scfacnew)--node[pos = 0.5,color = black, right] {$e'$} (45:\scfacnew)

};
\draw[line width = \scfac*0.015, color = black!50]{
(0,0) node[circle, draw, fill=black!10,inner sep=\scfac*0.015, minimum width=\scfac*0.08] {}
(0,0) node[anchor=south,yshift=-\scfac*0.25,color = black]{$v$}
(135+0:\scfacnew) node[circle, draw, fill=black!10,inner sep=\scfac*0.015, minimum width=\scfac*0.08] {}
(135+90:\scfacnew) node[circle, draw, fill=black!10,inner sep=\scfac*0.015, minimum width=\scfac*0.08] {}
(-45:\scfacnew) node[circle, draw, fill=black!10,inner sep=\scfac*0.015, minimum width=\scfac*0.08] {}
(45:0.5*\scfacnew)+(-45:0.5*\scfacnew) node[circle, draw, fill=black!10,inner sep=\scfac*0.015, minimum width=\scfac*0.08] {}
(45:\scfacnew) node[circle, draw, fill=black!10,inner sep=\scfac*0.015, minimum width=\scfac*0.08] {}
};
\end{tikzpicture}
\caption[]{The graph $G$}
\label{fig:graph G}
\end{minipage}
\hspace*{0pt}
\begin{minipage}[b][5cm]{.44\textwidth}
\centering
\begin{tikzpicture}[thick,scale= .7175]
  \pgfmathsetlengthmacro\scfac{2.5cm}
  \pgfmathsetlengthmacro\scfacnew{1.41*\scfac}
  \pgfmathsetmacro{\sepang}{120}
  \pgfmathtruncatemacro{\c}{3}
\draw[line width = \scfac*0.02, color = blue]{
(0,0) -- node[pos = 0.6, left] {7} (0+90:\scfac)
(0,0) -- node[pos = 0.5, below ] {1} (90+90:\scfac)
(0,0) -- node[pos = 0.5, left] {9} (180+90:\scfac)
(0,0) -- node[pos = 0.6, below] {4} (270+90:\scfac)

(0+90:\scfac) -- node[pos = 0.5, above left] {5} (90+90:\scfac)
[bend left] (0+90:\scfac) to node[pos = 0.6,below right] {3} (180+90:\scfac)
(0+90:\scfac) -- node[pos = 0.5, above right] {10} (270+90:\scfac)

(90+90:\scfac) -- node[pos = 0.5, below left] {11} (180+90:\scfac)
[bend left] (90+90:\scfac) to node[pos = 0.25, above right] {8} (270+90:\scfac)

(180+90:\scfac) -- node[pos = 0.5, below right] {6} (270+90:\scfac)

(270+90:\scfac) -- node[pos = 0.5, above left] {12} ($(45:\scfac)+(270+90:\scfac)$)
(270+90:\scfac) -- node[pos = 0.5, above] {2} ($(0:\scfac)+(270+90:\scfac)$)
(270+90:\scfac) -- node[pos = 0.5, below left] {0} ($(-45:\scfac)+(270+90:\scfac)$)
};
\draw[line width = \scfac*0.015, color = black!50]{
(0+90:\scfac) node[circle, draw, fill=black!10,inner sep=\scfac*0.015, minimum width=\scfac*0.08] {}
(90+90:\scfac) node[circle, draw, fill=black!10,inner sep=\scfac*0.015, minimum width=\scfac*0.08] {}
(180+90:\scfac) node[circle, draw, fill=black!10,inner sep=\scfac*0.015, minimum width=\scfac*0.08] {}
(270+90:\scfac) node[circle, draw, fill=black!10,inner sep=\scfac*0.015, minimum width=\scfac*0.08] {}
(270+90:\scfac) node[anchor=south,yshift=-\scfac*0.3,color = black]{$v'$}
(0,0) node[circle, draw, fill=black!10,inner sep=\scfac*0.015, minimum width=\scfac*0.08] {}
(45:\scfac)+ (270+90:\scfac) node[circle, draw, fill=black!10,inner sep=\scfac*0.015, minimum width=\scfac*0.08] {}
(0:\scfac)+(270+90:\scfac) node[circle, draw, fill=black!10,inner sep=\scfac*0.015, minimum width=\scfac*0.08] {}
(-45:\scfac)+(270+90:\scfac) node[circle, draw, fill=black!10,inner sep=\scfac*0.015, minimum width=\scfac*0.08] {}
};
\end{tikzpicture}
\caption[]{The graph $H$}
\label{fig:graph H}
\end{minipage}
\end{figure}
First, we check that $cms(G) = 1$.
Suppose otherwise, that $cms(\ell)=2$ for some ordering $\ell$ of~$G$.
As $G$ has $6$ edges and the vertex $v$ has degree $3$, the edges incident with $v$
are, without loss of generality, labelled as depicted in Figure~\ref{fig:graph G}.
However, for any choice of a label for the edge $e$,
there will be two cyclically consecutive edges incident with a common vertex.
Thus, $cms(G) = 1$.
On the other hand, it is easy to check that,
for any ordering $\ell$ of $G$ with the edges incident
with $v$ labelled as depicted, $cms_4(\ell) \geq 8$.
As $\Delta(G) = 3$ and $|E(G)|=6$,
the lower bound of Lemma~\ref{lem: seqity gen r} for $G$ when $r=4$ is
$1 \times 6 + cms_1(G) = 7 < 8 \leq cms_4(G)$.
By similar reasoning,
the graph $G'$ obtained from $G$ by removing the edge~$e'$ satisfies
$ms(G')=1$ and $ms_4(G') \geq 7$,
which is strictly above the lower bound given by Lemma~\ref{lem: seqity gen r}.
The bounds in Lemma~\ref{lem: seqity gen r} are thus not always achieved.

We can also show that Lemma~\ref{lem: multi from non-multi}
is no longer true if cyclic-sequencibility is replaced by non-cyclic sequencibility.
Consider the graph $H$ in Figure~\ref{fig:graph H}.
It is easy to verify that the ordering $\ell$ of $H$
depicted in Figure~\ref{fig:graph H} satisfies $ms(\ell) = 2$
and, in particular, that $ms(G) \geq 2$.
The graph $2H$ has 24 edges, 14 of which are incident with~$v$.
Therefore, for any ordering $\ell'$ of $2H$ corresponding to the sequence of edges
$e_0, \ldots, e_{23}$,
at least one of the 12 pairs of edges $e_{2i},e_{2i+1}$ for $i \in [12]$
has both of its edges incident with $v$, by the Pigeonhole Principle.
Thus, no ordering $\ell'$ of $2H$ can satisfy $ms(\ell') \geq 2$,
and so $ms(2H)=1 < 2 = ms(H)$.
So, there is no non-cyclic sequencibility analogue of Lemma~\ref{lem: multi from non-multi}.

We end the paper with the following conjecture on the matching sequencibility of
complete multi-partite graphs.
Let $K_{s(n)}$ be the complete 
$s$-partite graph with parts of size $n$.
\begin{conjecture}
For any integers $n \geq 2$ and $s\geq 2$,
\[
ms(K_{s(n)}) = cms(K_{s(n)}) =
\left\lfloor \frac{sn}{2} \right\rfloor -1 \,.
\]
\end{conjecture}

\section*{Acknowledgements}
I thank my supervisor Thomas Britz for very helpful discussions during the writing of this paper.

\end{document}